\documentclass[11pt,letterpaper]{amsart}

\usepackage{amsmath}               
\usepackage[
colorlinks=true,              
linkcolor=blue,                
citecolor=red,               
urlcolor=cyan                  
]{hyperref}

\numberwithin{equation}{section}
\theoremstyle{plain}
\newtheorem{theorem}{Theorem}[section]
\newtheorem{prop}[theorem]{Proposition}

\newtheorem{lemma}[theorem]{Lemma}
\newtheorem{conjecture}{Conjecture}

\newtheorem{corollary}[theorem]{Corollary}

\theoremstyle{definition}

\newtheorem{remark}[theorem]{Remark}

\newcommand{\Rmnum}[1]{\expandafter\@slowromancap\romannumeral #1@}

\newcommand{\mr}{\mathbb{R}}
\newcommand{\ud}{\mathrm{d}}
\newcommand{\ms}{\mathbb{S}}
\newcommand{\fint}{-\mkern -19mu\int}
\allowdisplaybreaks

\address{Mingxiang Li, Department  of Mathematics \& Institute of Mathematical Sciences, The Chinese University of Hong Kong, Shatin, NT, Hong Kong}
\email{mingxiangli@cuhk.edu.hk}

\address{X. Xu,   School  of Mathematics, Nanjing University, China }
\email{matxuxw@nju.edu.cn}

\begin{document}
	\title[Isoperimetric inequality about  $Q$-curvature]{A sharp isoperimetric inequality and the top order $Q$-curvature}
	\author{Mingxiang Li, Xingwang Xu}
	\date{}
	\maketitle
	\begin{abstract}
		For a smooth, complete  and  normal metric $g = e^{2u}|dx|^2$  with finite total $n$-th order $Q$-curvature on $\mathbb{R}^n$ with dimension $n \geq 2$, we first show that  everywhere non-negativity (resp. non‑positivity) $n$-th order $Q$-curvature $Q_g^{(n)}$ implies everywhere non-negativity (resp. non‑positivity) of the sectional curvature. Based on this fact,  we secondly show that, once $Q_g^{(n)}$ is non-negative, then for any compact domain $\Omega \subset \mathbb{R}^n$ with  smooth boundary $\partial\Omega$, the following sharp isoperimetric  inequality holds:
		\[
		|\partial\Omega|_g^{\frac{n}{n-1}} \geq n^{\frac{n}{n-1}} |\mathbb{B}^n|^{\frac{1}{n-1}} \left(1 - \frac{2}{(n-1)!\,|\mathbb{S}^n|} \int_{\mathbb{R}^n} Q_g^{(n)} \, d\mu_g\right) |\Omega|_g.
		\]
The third claim in this article is that, if the $n$-th order $Q$-curvature, $Q_g^{(n)}$, is non-positive and  under the main assumption that Cartan-Hardamard conjecture holds true, then we have the sharp inequality
			\[
		|\partial\Omega|_g^{\frac{n}{n-1}} \geq n^{\frac{n}{n-1}} |\mathbb{B}^n|^{\frac{1}{n-1}}|\Omega|_g.
		\]
			\end{abstract}
\section{Introduction}

The isoperimetric inequality is one of the most ancient topics in mathematics. It has been widely used in analysis, geometry, physics and many other fields.  A wonderful survey article \cite{Oss} by Osserman is recommended for further discussion.  Let us start the story  with dimension two.  For any compact domain $\Omega$ with  smooth boundary
$\partial \Omega$ in the two dimensional  Euclidean space, the classical and famous isoperimetric inequality can be written as   $$|\partial\Omega|^2\geq 4\pi|\Omega|$$  
where $|\cdot|$ denotes the measure of a set with respect to an appropriate measure element. When such an inequality is considered on a surface, the curvature of the metric naturally plays a role.   Based on the seminal work of Cohn-Vossen \cite{CV} for complete  open surfaces with Gaussian curvature $K_g\geq 0$,  Fiala \cite{Fiala} showed that, for any compact domain $\Omega$ in a simply connected and analytic open surface $(M^2,g)$, the following inequality holds
\begin{equation}\label{Fiala's inf}
	|\partial\Omega|_g^2\geq 4\pi\left(1-\frac{1}{2\pi}\int_{M^2}K_g\ud\mu_g\right)|\Omega|_g.
\end{equation}
It is worth to point out that $\int_{M^2}K_g\ud\mu_g\leq 2\pi$ is guaranteed by the  classical Cohn-Vossen inequality \cite{CV}, \cite{Finn} and \cite{Huber 57}. Later, Huber \cite{Huber 54} used potential theory for subharmonic functions to generalize this result by removing the analytic condition and assuming only that the negative part of the Gaussian curvature is integrable:
\begin{equation}\label{Fiala-Huber inequality}
		|\partial\Omega|_g^2\geq 4\pi\left(1-\frac{1}{2\pi}\int_{M^2}K^+_g\ud\mu_g\right)|\Omega|_g
\end{equation}
where $K_g^+$ denotes the positive part of Gaussian curvature $K_g$.  For Gaussian curvature  satisfying $K_g\leq K_0$ for some constant $K_0$, Bol \cite{Bol}  established the following isoperimetric inequality 
$$|\partial \Omega|_g^2\geq 4\pi |\Omega|_g-K_0|\Omega|_g^2$$
which plays an important role in the remarkable sphere covering inequality of Gui and Moradifam  \cite{Gui-Mora}.  Other isoperimetric comparison theorems  on rotational symmetric surfaces have been  established by Benjamini and Cao \cite{BC} and by Topping \cite{Topping} where the mean curvature flow is involved.

In  higher dimensional Euclidean space, the classical isoperimetric inequality remains valid: for any compact domain $\Omega\subset \mr^n$, we have
\begin{equation}\label{sharp iso}
	|\partial\Omega|^{\frac{n}{n-1}}\geq n^{\frac{n}{n-1}}|\mathbb{B}^n|^{\frac{1}{n-1}}|\Omega|
\end{equation}
where $|\mathbb{B}^n|$ denotes the volume of standard unit ball in $\mr^n$. However,  the corresponding inequality on general manifolds seems considerably difficult.  For example, on  $n$-dimensional  manifolds equipped with non-positive sectional curvature metric $g$,  the  \emph{Cartan-Hadamard conjecture} states that  the sharp isoperimetric inequality \eqref{sharp iso} holds. This conjecture has been verified  for $n=2$ (Weil \cite{Weil}, Beckenbach-Radó  \cite{BR}), $n=3$ (Kleiner \cite{Kleiner}), and $n=4$ (Croke \cite{Croke}).  For $n\geq 5$, it  remains open.  More recently development on this issue can be found in \cite{GS}.   On the other hand, under the assumption $Ric_g\geq (n-1)g$, the corresponding  iso-perimetric inequality is known as L\'evy-Gromov inequality, its recent development can be found in the work of Cavalletti-Mondino \cite{Cava-Modi}.

In recent years,  the isoperimetric inequality has been proved  for complete non‑compact manifolds with non‑negative Ricci curvature.  For three-dimensional case, Agostiniani, Fogagnolo, and  Mazzieri \cite{AFM}, combined  the Willmore-type inequality with mean curvature flow due to Huisken,  established a sharp isoperimetric inequality involving the asymptotic volume ratio $\mathrm{AVR}(g)$, defined by,  for an $n$-dimensional manifold,
$$\mathrm{AVR}(g):=\lim_{r\to\infty}\frac{|\{p\in M^n:d_g(p,q)<r\}|_g}{|\mathbb{B}^n|r^n}$$
for some fixed point $q\in M^n$.  The existence of this  limit is guaranteed  by Bishop-Gromov volume comparison theorem. The reader is referred to Schulze's work \cite{Schulze} for the  discussion on isoperimetric inequalities  and mean curvature flow.  More recently, the ABP method is applied by Brendle \cite{Brendle} to set up the  isoperimetric inequality on manifolds with non-negative Ricci curvature in  all dimensions: for any compact domain $\Omega$ with boundary $\partial \Omega$, we have
\begin{equation}\label{brendle's inequality}
	|\partial \Omega|_g^{\frac{n}{n-1}}\geq n^{\frac{n}{n-1}}|\mathbb{B}^n|^{\frac{1}{n-1}}(\mathrm{AVR}(g))^{\frac{1}{n-1}}|\Omega|_g.
\end{equation}
In \cite{BK}, Balogh and Krist\'aly proved the inequality \eqref{brendle's inequality} on $\mathrm{CD}(0,N)$ metric measure spaces via the  Brunn-Minkowski inequality and optimal mass transport theory. Using  this inequality together with  symmetrization techniques, they further established  a sharp $L^p$-Sobolev inequality (See Theorem 1.2 in \cite{BK}).

Now let us return to our concern,  in terms of conformal geometry,  $Q$-curvature is a natural higher‑dimensional analogue of Gaussian curvature and serves as a central concept in this field. Clearly the book written by Fefferman and Graham \cite{FG} is a good reference for $Q$-curvature.  In this respect, A. Chang once asked in \cite{Peter} whether there exists an isoperimetric inequality associated with $Q$-curvature or its integral. To explore the connection between isoperimetric inequalities and the integral of $Q$-curvature, we should begin with the remarkable work of Chang, Qing, and Yang \cite{CQY Duke}, which extended the classical Cohn–Vossen inequality \cite{CV} and Finn's identity \cite{Finn} to the four‑dimensional conformally flat manifold in terms of  $Q$-curvature. Before presenting their precise statement, we need to settle down some notations.

As promised, our aim is to establish a connection between the $n$-th order 
$Q$-curvature and the isoperimetric ratio, within the setting of smooth conformal metrics on $\mr^n$ with $n\geq 2$.
In this particular case, the conformal metric has a simple form: $g = e^{2u}  |dx|^2$ and the  $n$-th order $Q$-curvature can be simply defined by
  \begin{equation}\label{Q-def}
  	Q_g^{(n)}  =  e^{-nu}  (-\Delta)^{\frac{n}{2}}u.
  \end{equation}
  Observe that when $n$ is odd, the operator  $(-\Delta)^{\frac{n}{2}}$ is  a  non-local operator,  and to make it well define, one does requires extra conditions; a good reference  \cite{CS, Chang-Gonz} is recommended for the detailed treatment. Since our treatment is just mainly concerned with the integral representation (i.e, normal metric), we do assume this $n$-th order $Q$-curvature is well defined.

  Recall that a  conformal metric $g = e^{2u} |dx|^2$ on $\mr^n$ is said to be of the finite total $Q$-curvature if $Q_g^{(n)}e^{nu}\in L^1(\mr^n)$ and  a metric with finite total $Q$-curvature is called normal  if the conformal factor $u(x)$ satisfies the integral representation:
  \begin{equation}\label{normal solution}
  	u(x)=\frac{2}{(n-1)!|\mathbb{S}^n|}\int_{\mr^n}\log\frac{|y|}{|x-y|}Q_g^{(n)}(y)e^{nu(y)}\ud y+C
  \end{equation}
  where $|\mathbb{S}^n|$ denotes the volume of  standard n-sphere, and $C$ denotes a constant. 
  
Thanks to the combined efforts of Cohn-Vossen \cite{CV}, Huber \cite{Huber 57}, Chang-Qing-Yang \cite{CQY Duke}, Fang \cite{Fang}, and Ndiaye-Xiao \cite{NX}, it is now known that for a complete normal metric 
 $g=e^{2u}|dx|^2$ on $\mr^n$,  the following generalized Cohn-Vossen type inequality holds:
  \begin{equation}\label{CV for Q}
  	\int_{\mr^n}Q_g^{(n)}\ud\mu_g\leq \frac{(n-1)!|\mathbb{S}^n|}{2}
  \end{equation}
  where $\ud\mu_g:=e^{nu}\ud x$.  As Finn \cite{Finn}  figured out the deficit in Cohn-Vossen's inequality to get his identity, they \cite{CQY Duke, NX} also found the deficit in above inequality \eqref{CV for Q} which is the isoperimetric ratio near infinity so that they also have the identity:
  \begin{equation}\label{Finn for Q}
  	1-\frac{2}{(n-1)!|\mathbb{S}^n|}\int_{\mr^n}Q_g^{(n)}\ud\mu_g=\lim_{r\to\infty}\frac{|\partial B_r(0)|_g^{\frac{n}{n-1}}}{n^{\frac{n}{n-1}}|\mathbb{B}^n|^{\frac{1}{n-1}}|B_r(0)|_g}
  \end{equation}
where $B_r(0)$ denotes the Euclidean ball of radius $r$   in $\mr^n$  centered at the  origin. For brevity, we use $\alpha_0$ to denote the normalized total $Q$-curvature:
\begin{equation}\label{alpha_0 def}
	\alpha_0:=\frac{2}{(n-1)!|\mathbb{S}^n|}\int_{\mr^n}Q_g^{(n)}\ud\mu_g
\end{equation}
and will be used throughout this paper.

  The identity \eqref{Finn for Q} seems indicate that the isoperimetric inequality  is necessarily associated with the integral of $Q$-curvature, at least when the domain is very large. However, it is still mysterious whether such a relationship holds for all compact domains. For brevity, we set up some notations first: given a compact domain $\Omega\subset\mr^n$ with boundary $\partial \Omega$, we define the  isopermetric quotient as follows
  \begin{equation*}
  	I_g(\Omega):=\frac{|\partial \Omega|_g^{\frac{n}{n-1}}}{n^{\frac{n}{n-1}}|\mathbb{B}^n|^{\frac{1}{n-1}}|\Omega|_g},
  \end{equation*}
  and define the isoperimetric ratio over all such domains:
  \begin{equation}\label{I_g def}
  	I_g:=  \inf_{\Omega} I_g(\Omega).
  \end{equation}
As pointed out in \cite{CQY Duke}, without the normal metric assumption, \eqref{Finn for Q} fails in general and the integral of the $Q$-curvature can be arbitrary. This is why we work with normal metrics throughout this paper.

An immediate consequence of \eqref{Finn for Q} can be stated as follows:   for a complete  normal metric on $\mr^n$, there holds
$$ I_g\leq 1-\alpha_0.$$
The natural question is how to characterize the lower bound of $I_g$. Even for those  complete normal  metrics on $\mr^n$, the question is far away from the satisfactory answer. Maybe the first answer to it is done by Bonk, Heinonen, and Saksman \cite{BHS}, who showed that, if the  normal metric has sufficiently small total 
$Q$-curvature, then $I_g>0$.  Soon after, Xiao \cite{Xiao} also proved that $I_g > 0$, under different assumptions. Several years later, Wang \cite{Wang IMRN} showed that  $I_g>0$ for the normal metric under the optimal  assumption $\alpha_0<1$,  which can be thought of the higher dimensional generalization of the work of Li and Tam \cite{Li-Tam}.

 Up to now, the lower bounds of $I_g$ obtained so far are dependent of the metric $g$ itself and on the dimension $n$. It was Wang \cite{Wang Adv} who realized that the lower bound of $I_g$  actually depends on the integral of the positive and negative parts of the top order $Q$-curvature as well as the dimension $n$.  Since this is closely related to our work here, let us put some details here.
 First, we set up some notations: 
\[
\alpha_0^+:=\frac{2}{(n-1)!\,|\mathbb{S}^n|}\int_{\mathbb{R}^n} \bigl(Q^{(n)}_g\bigr)^+ \ud\mu_g , \;
\alpha_0^-:=\frac{2}{(n-1)!\,|\mathbb{S}^n|}\int_{\mathbb{R}^n} \bigl(Q^{(n)}_g\bigr)^- \ud\mu_g
\]
where $\varphi^+$ and $\varphi^-$ denote the positive and negative parts of the function $\varphi$.
 Wang's statement in \cite{Wang Adv} can be restated as: if one assumes $\alpha_0^+<1$ and $\alpha_0^-<+\infty$, then
\[
I_g \geq C(n,\alpha_0^+,\alpha_0^-)>0
\]
for some positive constant $C(n,\alpha_0^+,\alpha_0^-)$ depending only $n,\alpha_0^+$ and $\alpha_0^-$.

Recall that Fiala has been shown more than we just mentioned above  \eqref{Fiala's inf} which now we call the Fiala's identity which says, in two dimensioal case, $I_g = 1 - \alpha_0$.  Therefore  it is natural to ask ourself, is this always true for higher dimensional situation? The main purpose of this article is to show that indeed $ I_g = 1 - \alpha_0$  for normal metrics with the condition that $Q_g^{(n)} \geq 0$.

\begin{theorem}\label{thm:main theorem}
Let $(\mr^n, g)$ be a smooth complete manifold with conformal normal metric $g=e^{2u}|dx|^2$ and integer $n\geq 2$. Suppose that its $n$-th order $Q$-curvature is non-negative everywhere.  Then, the  isoperimetric ratio $I_g$ \eqref{I_g def} is given by $1 - \alpha_0$, that is,  $$I_g=1-\frac{2}{(n-1)!|\mathbb{S}^n|}\int_{\mr^n}Q_g^{(n)}\ud \mu_g.$$
\end{theorem}

There are several motivations for us to consider above result: first of all, it comes to generalize the corresponding result in two  dimensional case; secondly we already set the  connection between the non-negativity of the top order $Q$-curvature and  non-negativity of Ricci curvature of the metric in our recent work joint with J. Wei \cite{LWX}. By rechecking the idea from there,  recently we realize that, in fact, more can be done. In particular,  for a smooth complete  normal metric,  the sign of the top order  $Q$-curvature actually  determines the sign of sectional curvature of $g$.  For a Riemannian manifold $(M^n,g)$, the lower bound of sectional curvature is defined by
\[
\operatorname{sec}_g(p) := \inf_{\substack{E_1,E_2\in T_pM \\ |E_1|=|E_2|=1,\; g(E_1,E_2)=0}} g(R(E_1,E_2)E_2,E_1),
\]
where $R$ denotes the Riemann curvature tensor of $g$. We state it in the following theorem.

\begin{theorem}\label{thm:positive sectional for normal metric}
Let $(\mr^n, g)$ be  a smooth complete manifold with conformal normal metric $g=e^{2u}|dx|^2$ and integer $n\geq2$.
	\begin{enumerate}
		\item If the top order $Q$-curvature $Q_g^{(n)}$ is non-negative everywhere on $\mr^n$, then its sectional curvature $sec_g$ must be  non-negative;
		\item If the top order  $Q$-curvature $Q_g^{(n)}$  is non-positive everywhere, then its sectional  curvature  $sec_g$ must be non-positive.
	\end{enumerate}
\end{theorem}

If we assume  the Cartan–Hadamard conjecture holds, we should have the following conclusion concerning isoperimetric ratio in non-positive  case:
\begin{theorem}\label{thm:Q_g leq0}
Consider  a smooth complete  normal metric $g=e^{2u}|dx|^2$  on $\mr^n$    with
dimension $n\geq 2$ and $Q^{(n)}_g\leq 0.$  Assume that  the Cartan–Hadamard conjecture is true for all dimension.   Then the  isoperimetric ratio 
$I_g=1.$
\end{theorem}

This paper is organized as follows.  Section \ref{sec: noraml and behavior} is devoted to establishing some elementary estimates on the asymptotic behavior of conformal factor of normal metric,  most of them has  appeared in the previous works of the first author \cite{Li Adv, Li conformal} more or less. We include the proof for self-containment.  In Section \ref{sec:Q and sec}, we explore the closed relationship between the top $Q$-curvature and sectional curvature and prove Theorem \ref{thm:positive sectional for normal metric}. The purpose of Section \ref{sec:volume ratio} is to show that the volume ratio near infinity has a sharp lower bound related to the integral of the 
$n$-th order $Q$-curvature.  With the help of these estimates in previous several sections and Brendle's isoperimetric inequality \eqref{brendle's inequality}, we complete the proof of Theorems \ref{thm:main theorem}, \ref{thm:Q_g leq0}  in Section \ref{sec:proof of main theorem}.

\hspace{3em}

{\bf Acknowledgment.} The first author would like to   thank Professor Mijia Lai for helpful discussions and for inspiring him to consider Theorem \ref{thm:positive sectional for normal metric}.

\section{Normal metric and its asymptotic behavior}\label{sec: noraml and behavior}

First, we recall the equation \eqref{Q-def}. When $n$ is even, $Q^{(n)}_g$ can be defined through the local equation. 
When $n$ is odd, the operator is interpreted as 
\[
(-\Delta)^{\frac{n}{2}} = (-\Delta)^{\frac{1}{2}} \circ (-\Delta)^{\frac{n-1}{2}}.
\]
To define the fractional operator $(-\Delta)^{\frac{1}{2}}$, we introduce the space 
\[
L_{1/2}(\mathbb{R}^n) := \left\{ \varphi \in L^1_{\mathrm{loc}}(\mathbb{R}^n) \;\middle|\; \int_{\mathbb{R}^n} \frac{|\varphi(x)|}{1+|x|^{n+1}} \, dx < +\infty \right\}.
\]
Then, for some constant $C(n)$ depending  on $n$,
\[
(-\Delta)^{\frac{1}{2}} f(x) := C(n) \, \mathrm{P.V.} \int_{\mathbb{R}^n} \frac{f(x)-f(y)}{|x-y|^{n+1}} \, dy.
\]
More discussion about the fractional Laplacian can be found in \cite{CS} and \cite{Chang-Gonz}. 
Thus, for  odd $n$, in order to make sure the $n$-th order $Q$-curvature $Q^{(n)}_g$ of $g = e^{2u}|dx|^2$ to be well-defined, we need to assume that $(-\Delta)^{\frac{n-1}{2}} u \in L_{1/2}(\mathbb{R}^n)$.  
For all $n \ge 2$, the following identity holds:
\begin{equation}\label{Green function}
	(-\Delta)^{\frac{n}{2}} \left( \frac{2}{(n-1)!\,|\mathbb{S}^n|} \log\frac{1}{|x|} \right) = \delta_0(x),
\end{equation}
where $\delta_0(x)$ denotes the Dirac operator. From this,  the normal solution \eqref{normal solution}  means that $u$ satisfying the equation \eqref{Q-def} can be represented by  a integral equation via its Green's function.

Some of the estimates below have been established in \cite{Li Adv, Li conformal} by the first  author.   For the reader's convenience, we sketch the proofs.

Throughout this section, we focus the smooth normal metric $g = e^{2u}|dx|^2$ on $\mathbb{R}^n$ with dimension $n\geq 2$.

For brevity, we should use the short forms $\ud\nu(y)$  and $\ud|\nu|(y)$ to denote the  measures:
$$\ud\nu(y):=\frac{2}{(n-1)!|\mathbb{S}^n|}Q_g^{(n)}(y)e^{nu(y)}\ud y,$$
and
$$ \ud|\nu|(y):=\frac{2}{(n-1)!|\mathbb{S}^n|}|Q_g^{(n)}(y)|e^{nu(y)}\ud y.$$
For any measurable set $E\subset \mr^n$, we introduce the following notations
$$\fint_E\varphi\ud \mu:=\frac{1}{|E|}\int_E\varphi\ud \mu, \quad  \bar u(r):=\fint_{\partial B_r(0)}u\ud\sigma.$$
\begin{lemma} \label{lem:bar u}
For a smooth complete  normal metric $g=e^{2u}|dx|^2$ on $\mr^n$ with $n\geq 2$, the spherical average of $u$ has a nice behavior:
\[\bar u(r)=(-\alpha_0+o(1))\log r\]
where  $o(1)\to 0$ as $r\to\infty$.
\end{lemma}
\begin{proof}

Fix a point $x$ with $|x|\geq e^4$. Elementary calculation shows that  $|x|\geq 2\log|x|$.  Let us  split $\mr^n$ into three pieces
	$$A_1=B_{1}(x), \quad A_2=B_{\log|x|}(0),\quad A_3=\mr^n\backslash (A_1\cup A_2).$$
	
	We deal with $A_2$ first. For $y\in A_2$ and $|y|\geq2$ , it is rather easy to see that $|\log\frac{|x|\cdot|y|}{|x-y|}|\leq \log(2\log|x|)$.
	On the other hand,  for $|y|\leq 2$ and $y$ must be in $A_2$, we also have $ |\log\frac{|x|\cdot|y|}{|x-y|}|\leq |\log|y||+C$. Observe that the normal metric  assumption  implies $\int_{A_2} \ud|\nu|(y) \leq C $ for some constant $C > 0$. Thus we can easily conclude:
	
	\begin{align*}
			&\left|\int_{A_2}\log\frac{|y|}{|x-y|}\ud\nu(y)+\log|x|\int_{A_2}\ud\nu(y)\right|\\
			\leq &C\log\log|x|+C=o(1)\log|x|
	\end{align*}
where $o(1)\to 0$ 	as $|x|\to \infty$.
	
	Now we turn our attention to the set $A_3$. For $y\in A_3$,  a simple observation indicates that 
	$$\frac{1}{|x|+1}\leq\frac{|y|}{|x-y|}\leq |x|+1.$$ 
	With the help of this estimate, the integral over $A_3$ can be dominated as
	\begin{equation*}
		|\int_{A_3}\log\frac{|y|}{|x-y|}\ud\nu(y)|\leq \log(|x|+1)\int_{A_3}\ud|\nu|(y).
	\end{equation*}
	
	Finally we treat the integral over $A_1$:  by definition of $A_1$ , it is clear that $y\in B_1(x)$ implies that 
	$1\leq |y|\leq |x|+1$. Thus we can easily to control this term:
	
	$$|\int_{A_1}\log|y|\ud\nu(y)|\leq \log(|x|+1)\int_{A_1}\ud|\nu|(y).$$
		Due to $Q_g^{(n)}e^{nu}\in L^1(\mr^n)$, one has  $\int_{A_3\cup A_1}\ud|\nu|(y)\to 0$ as $|x|\to \infty$. Then,  
	$$\frac{2}{(n-1)!|\mathbb{S}^n|}\int_{A_2}Q_g^{(n)}(y)e^{nu(y)}\ud y=\alpha_0+o(1).$$
Using  these estimates, one obtains that 
	\begin{equation}\label{equ: u loc}
		u(x)=(-\alpha_0+o(1))\log|x|+\int_{B_1(x)}\log\frac{1}{|x-y|}\ud\nu(y).
	\end{equation}
For $r = |x| \gg1$, a direct computation yields that
	\begin{align*}
		&\left|\int_{\partial B_r(0)}\int_{B_1(x)}\left(\log \frac{1}{|x-y|}\right)\ud\nu(y)\ud\sigma(x)\right|\\
		\leq &\int_{\partial B_r(0)}\int_{B_{r+1}(0)\backslash B_{r-1}(0)}\left|\log |x-y|\right|\cdot \ud|\nu|(y)\ud\sigma(x)\\
		\leq &\int_{B_{r+1}(0)\backslash B_{r-1}(0)}\int_{\partial B_r(0)}\left|\log|x-y|\right|\ud\sigma(x)\ud |\nu|(y)\\
		\leq&\int_{B_{r+1}(0)\backslash B_{r-1}(0)}\int_{\partial B_r(0)\cap B_1(y)}\left|\log|x-y|\right|\ud\sigma(x)\ud|\nu|(y)\\
		&+\int_{B_{r+1}(0)\backslash B_{r-1}(0)}\int_{\partial B_r(0)\cap (\mr^n\backslash B_1(y))}\left|\log|x-y|\right|\ud\sigma(x)\ud |\nu|(y)\\
		\leq &\int_{B_{r+1}(0)\backslash B_{r-1}(0)}\left(C+Cr^{n-1}\log(2r+1)\right)\ud |\nu|(y),
	\end{align*}
	where $C$ is a positive constant. Note that from now on, we will use $C$ to denote a constant which may be different from line to line.
	
	Due to our assumption that $Q_g^{(n)}e^{nu}\in L^1(\mr^n)$, we can easily conclude:
	$$\fint_{\partial B_r(0)}\int_{B_1(x)}\left(\log \frac{1}{|x-y|}\right)\ud\nu(y)\ud\sigma(x)=o(1)\log r.$$
	Thus our claim follows by combining this with \eqref{equ: u loc}.  The proof of Lemma \ref{lem:bar u} is complete.
\end{proof}

The following calculus lemma will be useful later.

\begin{lemma}\label{lem:f average}
	Let $n$ be a integer $n\geq 3$ and $t$ be a real  number satisfying  $0<t\leq n-2$. For any $f\in L^1(\mr^n)\cap L_{loc}^\infty(\mr^n)$, there holds
	$$\lim_{r\to\infty}\fint_{\partial B_r(0)}\int_{\mr^n}\frac{|y|^t}{|x-y|^t}f(y)\ud y\ud\sigma(x)=0$$
	and
	$$\lim_{r\to\infty}\fint_{\partial B_r(0)}\int_{\mr^n}\frac{|x|^t}{|x-y|^t}f(y)\ud y\ud\sigma(x)=\int_{\mr^n}f(x)\ud x.$$
\end{lemma}
\begin{proof}
	Since $f\in L^1(\mr^n)$, for any $\epsilon>0$, there exists $R_\epsilon>0$ such that 
	\begin{equation}\label{small total f}
		\int_{\mr^n\setminus B_{R_\epsilon}(0)}|f|\ud x<\epsilon.
	\end{equation}
	
	With the help of Fubini's theorem, one has
	\begin{align*}
		&\left|\fint_{\partial B_r(0)}\int_{\mr^n}\frac{|y|^t}{|x-y|^t}f(y)\ud y\ud\sigma(x)\right|\\
		\leq &\int_{B_{R_\epsilon}(0)}|f(y)|\fint_{\partial B_r(0)}\frac{|y|^t}{|x-y|^t}\ud\sigma(x)\ud y\\
		&+\int_{\mr^n\setminus B_{R_\epsilon}(0)}|f(y)|\fint_{\partial B_r(0)}\frac{|y|^t}{|x-y|^t}\ud\sigma(x)\ud y\\
		=:&I_1(x)+I_2(x).
	\end{align*}
	
	By Lebesgue's dominated convergence theorem,  one can easily conclude that
	$$I_1(x)\to 0, \quad \mathrm{as}\; r\to\infty.$$
	
	First observe that, for $n\geq 3$, with the help of mean value property of harmonic functions, we obtain the formula
	\begin{equation}\label{LL identity}
		\fint_{\partial B_r(0)}\frac{1}{|x-y|^{n-2}}\ud\sigma(x)=\min\{r^{2-n},|y|^{2-n}\}.
	\end{equation}
	
	The identity \eqref{LL identity} and H\"older's inequality together imply that
	\begin{equation}\label{LL inequality}
		\fint_{\partial B_r(0)}\frac{|y|^t}{|x-y|^t}\ud\sigma(x)\leq \left(\fint_{\partial B_r(0)}\frac{|y|^{n-2}}{|x-y|^{n-2}}\ud\sigma(x)\right)^{\frac{t}{n-2}}\leq 1
	\end{equation}
	where the assumption $0<t\leq n-2$ is used.
	Combine this with the estimate \eqref{small total f} to get:
	$$I_2(x)<\epsilon.$$
	Due to the arbitrary choice of $\epsilon$, we have
	$$\lim_{r\to\infty}\fint_{\partial B_r(0)}\int_{\mr^n}\frac{|y|^t}{|x-y|^t}f(y)\ud y\ud\sigma(x)=0.$$
	This finishes the proof of the first part.
	
	To see second part, we first decompose $f$ into two parts: $f = f^+ - f^-$ where $f^+(y) = \max\{ 0, f(y)\}$ and $f^-(y) = \max\{ 0, - f(y)\}$. By similar reason as in the first part, there holds
	\begin{align*}
		&\fint_{\partial B_r(0)}\int_{\mr^n}\frac{|x|^t}{|x-y|^t}f^+(y)\ud y\ud\sigma(x)\\
		=&\int_{B_{R_\epsilon}(0)}f^+(y)\fint_{\partial B_r(0)}\frac{|x|^t}{|x-y|^t}\ud\sigma(x)\ud y\\
		&+\int_{\mr^n\setminus B_{R_\epsilon}(0)} f^+(y) \int_{\partial B_r(0)}\frac{|x|^t}{|x-y|^t}\ud\sigma(x)\ud y\\
		=:&I_3(x)+I_4(x).
	\end{align*}
	Again, by Lebesgue's dominated convergence theorem,  one gets:
	$$I_3(x)\to \int_{B_{R_\epsilon}(0)} f^+ \ud y$$
	as $r\to\infty.$
	
	Thus apply the estimate \eqref{LL inequality} and H\"older's inequality to obtain:
	$$\fint_{\partial B_r(0)}\frac{|x|^t}{|x-y|^t}\ud\sigma(x)\leq 1$$
	which yields that
	$$I_4(x)<\epsilon.$$
	Due to the arbitrary choice of $\epsilon$, one has
	$$\lim_{r\to\infty}\fint_{\partial B_r(0)}\int_{\mr^n}\frac{|x|^t}{|x-y|^t}f^+(y)\ud y\ud\sigma(x)=\int_{\mr^n}f^+(x)\ud x.$$
	Similarly, one has
	$$\lim_{r\to\infty}\fint_{\partial B_r(0)}\int_{\mr^n}\frac{|x|^t}{|x-y|^t}f^-(y)\ud y\ud\sigma(x)=\int_{\mr^n}f^-(x)\ud x.$$
	Hence by subtracting these two formulas, we obtain the desired estimate
	$$\lim_{r\to\infty}\fint_{\partial B_r(0)}\int_{\mr^n}\frac{|x|^t}{|x-y|^t}f(y)\ud y\ud\sigma(x)=\int_{\mr^n}f(x)\ud x.$$
	The proof is finished.
\end{proof}

\begin{lemma}\label{lem:ru'}
	For a smooth normal metric on $\mr^n$ with dimension $n \geq 2$, the spherical average of the conformal factor $u(x)$ has the first order asymptotic estimate:
	\[\lim_{r\to\infty} r\cdot \bar u'(r)=-\alpha_0\]
\end{lemma}	
\begin{proof}
	A direct computation yields that 
	\begin{equation}\label{r_bar u'r}
		r\cdot \bar u'(r) =	\fint_{\partial B_r(0)}r\cdot \frac{\partial u}{\partial r}\ud \sigma=\fint_{\partial B_r(0)}x\cdot \nabla u\ud \sigma.
	\end{equation}
	For $n=2$, apply the divergence theorem to obtain that 
	$$r\cdot \bar u'(r)=\frac{1}{2\pi}\int_{\partial B_r(0)}\frac{\partial u}{\partial r}\ud\sigma=\frac{1}{2\pi}\int_{B_r(0)}\Delta u\ud x$$
	which yields that
	$$r\cdot \bar u'(r)\to -\alpha_0, \quad \mathrm{as}\; r\to\infty.$$

	For $ n \geq 3$, using the integral representation of the conformal factor $u$, namely the equation \eqref{normal solution}, a direct computation yields that 
	\begin{equation}\label{x nabla u}
		x\cdot \nabla u(x)
		=-\int_{\mr^n}\frac{x\cdot(x-y)}{|x-y|^2}\ud\nu(y)
		=-\alpha_0+\int_{\mr^n}\frac{y\cdot(y-x)}{|x-y|^2}\ud \nu(y).
	\end{equation}
	Apply Lemma \ref{lem:f average} for $n\geq 3$ to conclude that
	\begin{equation*}
		\left|\fint_{\partial B_r(0)}\int_{\mr^n}\frac{y\cdot(y-x)}{|x-y|^2}\ud\nu(y)\ud\sigma(x)\right|\leq \fint_{\partial B_r(0)}\int_{\mr^n}\frac{|y|}{|x-y|}\ud|\nu|(y)\ud\sigma(x)\to 0
	\end{equation*}
	as $r\to\infty.$
	Thus, combining the above estimate with identities \eqref{r_bar u'r} and \eqref{x nabla u}, we finish the proof.
\end{proof}

\begin{lemma}\label{lem:|w|}
For a smooth complete normal metric $g=e^{2u}|dx|^2$ on $\mr^n$ with dimension $n \geq 2$, the conformal factor $u$ is almost radial in the $L^1$ sense near infinity, that is,
	$$\fint_{\partial B_r(0)}|u-\bar u|\ud\sigma\to 0, \quad \mathrm{as}\; r\to\infty.$$
\end{lemma}
\begin{proof}
	Based on the equation \eqref{normal solution}, one has
	$$u(x)-\bar u(r)=-\int_{\mr^n}\left(\log|x-y|-M_r(y)\right)\ud\nu(y)$$
	where 
	$$M_r(y)=\fint_{\partial B_r(0)}\log|z-y|\ud\sigma(z).$$
	
	Use the polar coordinate $x=r\xi$ where $\xi\in \ms^{n-1}$ and set $\eta=\frac{y}{r}$ to have
	$$\log|x-y|=\log r+\log|\xi-\eta|,$$
	$$M_r(y)=\log r+F(\eta), \quad F(\eta):=\frac{1}{|\ms^{n-1}|}\int_{\ms^{n-1}}\log|\theta-\eta|\ud\sigma(\theta).$$
	Then, apply Fubini's theorem to see that,  there holds
	$$\fint_{\partial B_r(0)}|u-\bar u|\ud\sigma\leq \int_{\mr^n}I_0(\eta)\ud|\nu|(y)$$
	where 
	$$I_0(\eta):=\frac{1}{|\ms^{n-1}|}\int_{\ms^{n-1}}\left|\log|\xi-\eta|-F(\eta)\right|\ud\sigma(\xi).$$
It is not hard to check that $I_0(\eta)$ is uniformly bounded and continuous.
For $|\eta|\geq \frac{3}{2}$ and $\xi, \theta\in \ms^{n-1}$, one has 
$$\frac{1}{5}\leq \frac{|\eta|-1}{|\eta|+1}\leq \frac{|\xi-\eta|}{|\theta-\eta|}\leq \frac{|\eta|+1}{|\eta|-1}\leq 5$$
which yields that
$\left|\log|\xi-\eta|-F(\eta)\right|\leq \log 5$ and then $I_0(\eta)\leq \log 5$ for $|\eta|\geq \frac{3}{2}$.  Notice that  $\frac{1}{2} \leq |\eta|\leq \frac{3}{2}$, for each $\eta$, we project it into $S^{n-1}$, that is, to write ${\bar \eta } = \frac{\eta}{|\eta|}$. Then a direct computation shows that  $I_0({\bar \eta})$ is uniformly bounded and $|\log| (\frac{1}{|\eta|} - 1)|$ is bounded.  This is enough to see that $I_0(\eta)$ is bounded on this region. Finally on the set $|\eta| \leq \frac{1}{2}$,  $|\log|\theta - \eta|| \leq \log 2$. This is enough to conclude that $I_0({\eta})$is uniformly bounded for all $\eta$. Besides,   one can easily see that  $\lim_{\eta\to 0}I_0(\eta)=0.$

With the estimate for $I_0(\eta)$, on one hand, we clearly have
$$\int_{\mr^n}I_0(\eta)\ud|\nu|(y)\leq C\int_{\mr^n}\ud|\nu|(y)\leq C.$$
On the other hand, for almost all $y$ in the $n$ dimensional measure, one has 
$$\lim_{r\to\infty}I_0\left(\frac{y}{r}\right)=0.$$
Thus, Lebesgue's dominated convergence theorem can be applied to conclude our claim  and to finish the proof of Lemma \ref{lem:|w|}.
\end{proof}

We continue to do the spherical average estimates of conformal factor $u$ in various form.

The following lemma plays an important role in our later argument. In fact, it has already played many roles in previous works. It was first proved by Finn~\cite{Finn} in the two-dimensional setting. Later, it was proved by Chang, Qing, and Yang~\cite{CQY Duke} in the four-dimensional form. We would like to point out that the method in fact works for all dimensions as long as the metric is normal. In \cite{Li conformal}, the first author gave a modified proof that is slightly different from earlier approaches. Although Lemma~2.7 in \cite{Li conformal} assumes $k>0$, the method also works for $k \le 0$.

\begin{lemma}\label{CQY lemma}
Let  $g=e^{2u}|dx|^2$ be a complete normal metric on $\mr^n$ with dimension $n \geq 2$.
	For any real number $k$, the function $u$  satisfies the equation \eqref{normal solution}. Then $e^{ku}$ is almost a spherical symmetric function near infinity, that is, 
	$$\fint_{\partial B_r(0)}e^{ku}\ud\sigma=e^{k\bar u(r)+o(1)}$$
	where $o(1)\to 0$ as $r\to\infty.$
\end{lemma}

\begin{proof}
	First, we split $u(x)$ into three parts
	\begin{align*}
	u(x)= &\int_{B_{\frac{|x|}{2}}(0)}\log\frac{|y|}{|x|}\ud \nu(y)+C\\
		&+\int_{B_{\frac{|x|}{2}}(0)}\log\frac{|x|}{|x-y|}\ud \nu(y)\\
		&+\int_{\mr^n\backslash B_{\frac{|x|}{2}}(0)}\log\frac{|y|}{|x-y|}\ud \nu(y)\\
		=&:D_1(x)+D_2(x)+D_3(x).
	\end{align*}

First it is obvious to see that $D_1(x)$ is radially symmetric. Thus,  for any $x\in \partial B_r(0)$, there holds
\begin{equation}\label{fint D_1}
	D_1(x)=\fint_{\partial B_r(0)}D_1\ud\sigma.
\end{equation}
	
	In order to treat the term $D_2(x)$, for $|x|\gg1$,  we split it into two sub-terms:
$$
		D_2(x)=\int_{B_{\sqrt{|x|}}(0)}\log\frac{|x|}{|x-y|}\ud \nu(y)
		+\int_{B_{\frac{|x|}{2}}(0)\backslash B_{\sqrt{|x|}}(0)}\log\frac{|x|}{|x-y|}\ud \nu(y).
$$
	By a direct computation, together with the fact that $Q^{(n)}_ge^{nu}\in L^1(\mr^n)$, one gets the estimate:
	$$\left|\int_{B_{\sqrt{|x|}}(0)}\log\frac{|x|}{|x-y|}\ud \nu(y)\right|\leq C\log\frac{\sqrt{|x|}}{\sqrt{|x|}-1}$$
	and 
	$$\left|\int_{B_{\frac{|x|}{2}}(0)\backslash B_{\sqrt{|x|}}(0)}\log\frac{|x|}{|x-y|}\ud \nu(y)\right|\leq C\int_{B_{\frac{|x|}{2}}(0)\backslash B_{\sqrt{|x|}}(0)}|Q^{(n)}|_ge^{nu}\ud y.$$
Since the right hand side of either term will tend to zero as  $|x|\to\infty$,  we obtain that
	\begin{equation}\label{D_2 to 0}
		D_2(x) = o(1),
	\end{equation}
	as $ r = |x| \to \infty.$
	
	Now we come to term $D_3(x)$. Using the elementary inequality $|t|\leq e^t+e^{-t}$ and Fubini's theorem, one can easily see that there holds
	\begin{align*}
		\left|\fint_{\partial B_r(0)}D_3\ud\sigma\right|\leq &C\int_{\mr^n\backslash B_{|x|/2}(0)}\fint_{\partial B_r(0)}|\log\frac{|y|}{|x-y|}|\ud\sigma\ud|\nu|(y)\\
		\leq &C\int_{\mr^n\backslash B_{|x|/2}(0)}\fint_{\partial B_r(0)}\left(\frac{|y|}{|x-y|}+\frac{|x-y|}{|y|}\right)\ud\sigma\ud|\nu|(y)\\
		\leq &C\int_{\mr^n\backslash B_{|x|/2}(0)}\fint_{\partial B_r(0)}\left(\frac{|y|}{|x-y|}+3\right)\ud\sigma\ud |\nu|(y)
	\end{align*}
	The estimate \eqref{LL inequality} gives arise
	\begin{equation*}
		\left|\fint_{\partial B_r(0)}D_3\ud\sigma\right|\leq C\int_{\mr^n\backslash B_{r/2}(0)}|Q^{(n)}_g|e^{nu}\ud y
	\end{equation*}
	which tends to zero as $r\to\infty$ since $Q^{(n)}_ge^{nu}\in L^1(\mr^n)$.
	Combine the identity \eqref{fint D_1} and  and the estimate \eqref{D_2 to 0} to have
	\begin{equation}\label{bar u=D_1+o(1)}
		u(x)-\bar u(|x|)=D_3(x)+o(1)
	\end{equation}
	where $o(1)\to 0$ as $r\to\infty$.
	
In order to further treat $D_3(x)$, we have to split it into also two terms
	\begin{align*}
		2kD_3(x)= &\frac{2}{(n-1)!|\mathbb{S}^n|}\int_{\mr^n\backslash B_{|x|/2}(0)}\log\frac{|y|}{|x-y|}(2kQ^{(n)}_g)^+e^{nu}\ud y\\
		&+\frac{2}{(n-1)!|\mathbb{S}^n|}\int_{\mr^n\backslash B_{|x|/2}(0)}\log\frac{|x-y|}{|y|}(2kQ^{(n)}_g)^-e^{nu}\ud y\\
		=&:E_1(x)+E_2(x).
	\end{align*}

Now by the assumption, for any small $\epsilon$ satisfying $0<\epsilon<\frac{1}{2}$,  there exists $r(\epsilon)>0$ such that,   if $|x|\geq r(\epsilon) $,
	$$\frac{2\|(2kQ^{(n)}_g)^+e^{nu}\|_{L^1(\mr^n\backslash B_{|x|/2}(0))}}{(n-1)!|\mathbb{S}^n|}\leq \epsilon.$$
If $\|(2kQ^{(n)}_g)^+e^{nu}\|_{L^1(\mr^n\backslash B_{|x|/2}(0))}>0$ and $r\geq r_1$,  denote by $(\ud y)^+$ the measure $\frac{(2kQ^{(n)}_g)^+e^{nu}\ud y}{\|(kQ^{(n)}_g)^+e^{nu}\|_{L^1(\mr^n\backslash B_{|x|/2}(0))}}$ and apply Jensen's inequality, together with Fubini's theorem to obtain:
	\begin{align*}
		&\fint_{\partial B_r(0)}e^{E_1(x)}	\ud\sigma\\
			\leq &\fint_{\partial B_r(0)}\int_{\mr^n\setminus B_{|x|/2}(0)}\left(\frac{|y|}{|x-y|}\right)^{\frac{2\|(2kQ^{(n)}_g)^+e^{nu}\|_{L^1(\mr^n\setminus B_{|x|/2}(0))}}{(n-1)!|\mathbb{S}^n|}} (\ud y)^+\ud\sigma\\
		\leq &\fint_{\partial B_r(0)}\int_{\mr^n\setminus B_{|x|/2}(0)}\left(\frac{|y|}{|x-y|}\right)^{\epsilon} (\ud y)^+\ud\sigma\\
		\leq & \int_{\mr^n\setminus B_{|x|/2}(0)}	\fint_{\partial B_r(0)}\left(\frac{|y|}{|x-y|}\right)^{\epsilon}\ud\sigma (\ud y)^+.
	\end{align*}
For $|y|\geq 2|x|$, one has
$$\left(\frac{|y|}{|x-y|}\right)^{p_1(x)}\leq 2^{\epsilon}.$$
For $|x|/2\leq |y|\leq 2|x|$, after suitable changing of variables, one has
$$
\fint_{\partial B_r(0)}\left(\frac{|y|}{|x-y|}\right)^{\epsilon}\ud\sigma(x)=\fint_{\partial B_1(0)}\left(\frac{|\xi_0|}{|\xi-\xi_0|}\right)^{\epsilon}\ud\sigma(\xi)$$
where $\frac{1}{2}\leq |\xi_0|\leq 2.$ In each  situation,  it not hard to check that 
$$\limsup_{\epsilon\to0}\fint_{\partial B_1(0)}\left(\frac{|\xi_0|}{|\xi-\xi_0|}\right)^{\epsilon}\ud\sigma(\xi)\leq 1.$$
Based on the arbitrary choice of $\epsilon$, one has 
	\begin{equation}\label{e^E_1}
	\limsup_{r\to\infty}\fint_{\partial B_r(0)}e^{E_1(x)}	\ud\sigma\leq 1.
\end{equation}	
In fact, 	for $n\geq 3$, by the estimate \eqref{LL inequality}, we can even obtain that 
	\begin{equation*}
		\fint_{\partial B_r(0)}e^{E_1(x)}	\ud\sigma\leq 1.
	\end{equation*}

	However, on the other hand, if $\|(2kQ^{(n)}_g)^+e^{nu}\|_{L^1(\mr^n\backslash B_{|x|/2}(0))}=0$, then the inequality \eqref{e^E_1} holds trivially.

	Now let us focus on the last term $E_2(x)$. First of all, it is easy check that for $y\in \mr^n\setminus B_{|x|/2}(0)$, the following estimate holds:
	$\frac{|x-y|}{|y|}\leq 3.$
	
	As in the case for $E_1$, we just need to deal with $\|(2kQ^{(n)}_g)^-e^{nu}\|_{L^1(\mr^n\backslash B_{|x|/2}(0))}>0$. 
	In this scenario,  we denote $\frac{(2kQ^{(n)}_g)^-e^{nu}}{\|(2kQ^{(n)}_g)^-e^{nu}\|_{L^1(\mr^n\setminus B_{|x|/2}(0))}} dy$ by $(\ud y)^-$
	\begin{align*}
		&\fint_{\partial B_r(0)}e^{E_2(x)}\ud\sigma\\
		\leq &\int_{\mr^n\setminus B_{|x|/2}(0)}\fint_{\partial B_r(0)}\left(\frac{|x-y|}{|y|}\right)^{\frac{2\|(2kQ^{(n)}_g)^-e^{nu}\|_{L^1(\mr^n\setminus B_{|x|/2}(0))}}{(n-1)!|\mathbb{S}^n|}}\ud\sigma (\ud y)^-\\
		\leq &3^{\frac{2\|(2kQ^{(n)}_g)^-e^{nu}\|_{L^1(\mr^n\setminus B_{|x|/2}(0))}}{(n-1)!|\mathbb{S}^n|}}
	\end{align*}
	which deduces that
	\begin{equation}\label{e^E_2}
		\limsup_{r\to\infty}\fint_{\partial B_r(0)}e^{E_2(x)}\ud\sigma\leq 1
	\end{equation}
since $\|(2kQ^{(n)}_g)^-e^{nu}\|_{L^1(\mr^n\backslash B_{|x|/2}(0))}\to 0$ as $|x|\to\infty.$
	H\"older's inequality, together with the estimates \eqref{e^E_1}, \eqref{e^E_2}, can be applied to obtain:
	$$
	\limsup_{r\to\infty} \fint_{\partial B_r(0)} e^{kD_3(x)}\ud\sigma\leq \limsup_{r\to\infty}\left(\fint_{\partial B_r(0)} e^{E_1(x)}\ud\sigma\right)^{\frac{1}{2}}\left(\fint_{\partial B_r(0)} e^{E_2(x)}\ud\sigma\right)^{\frac{1}{2}}\leq 1.
	$$
	Combine this estimate with the estimates \eqref{bar u=D_1+o(1)} to achieve:
	$$
	\limsup_{r\to\infty}\fint_{\partial B_r(0)} e^{ku-k\bar u}\ud\sigma=	\limsup_{r\to\infty}\fint_{\partial B_r(0)} e^{kD_3}\ud\sigma
	\leq 1.$$ 
	Then the following estimate is just the application of Jensen's inequality:
	$$ \fint_{\partial B_r(0)} e^{ku}\ud\sigma\geq e^{k\bar u}.$$
	
	Thus, we reach at
	$$\lim_{r\to\infty}\fint_{\partial B_r(0)} e^{ku-k\bar u}\ud\sigma=1.$$
	
	This is enough to close the proof.
\end{proof}

From now on, for our easy presentation, we set  $$\bar g: = e^{2\bar u}|dx|^2, \quad \ud \mu_{\bar g}:=e^{n\bar u}\ud x$$ and $$w(x):=u(x)-\bar u(|x|)$$ in the rest of article.

As a last piece of preparation, we show the following two lemmas.

\begin{lemma}\label{lem:e^kw}
		For a smooth complete normal metric $g=e^{2u}|dx|^2$ on $\mr^n$ with $n\geq 2$ and  $\alpha_0<1$, there holds, 
	$$\lim_{r\to\infty}\frac{\int_{B_r(0)}e^{ k w}\ud \mu_{\bar g}}{\int_{B_r(0)}\ud\mu_{\bar g}}=1,$$
	where $ k\in \mr$ is a fixed constant.
\end{lemma}
\begin{proof}
	Duo to the assumption that $\alpha_0<1$, by Lemma \ref{lem:bar u}, it is easy to show that
 \begin{equation}\label{infinte V_bar g}
 	\int_{B_r(0)}e^{n\bar u}\ud x\to\infty, \quad \mathrm{as}\; r\to\infty.
 \end{equation}
This estimate, together with Lemma \ref{CQY lemma} and L'H\^opital's rule, can be used to conclude that 
	\begin{align*}
		&\lim_{r\to\infty}\frac{\int_{B_r(0)}e^{ k w}\ud \mu_{\bar g}}{\int_{B_r(0)}\ud\mu_{\bar g}}\\
		=&\lim_{r\to\infty}\frac{\int_{\partial B_r(0)}e^{kw}e^{n\bar u}\ud\sigma}{\int_{\partial B_r(0)}e^{n\bar u(r)}\ud\sigma}\\
		=&\lim_{r\to\infty}\fint_{\partial B_r(0)} e^{kw}\ud\sigma=1.
	\end{align*}
\end{proof}

The following lemma concerning the radially symmetric metric $\bar g = e^{2\bar u}|dx|^2$ is easy but important. Therefore, we state it as a separated lemma.

\begin{lemma}\label{lem: bar g}
For a smooth complete normal metric $g=e^{2u}|dx|^2$ on $\mr^n$ with $n\geq 2$, the corresponding radially symmetric metric $\bar g=e^{2\bar u}|dx|^2$ is also a complete, normal metric. Moreover,  they share that same integral of $n$-th order $Q$-curvature.
\end{lemma}
\begin{proof}
The proof follows closely with the argument of Lemma 3.6 in \cite{CQY Duke}.
By the definition of distance function and the polar coordinate,  for any $r>0$ and $\theta\in \ms^{n-1}$, one has 
	$$d_g(0, r\theta)\leq \int^{r}_0e^{u(t\theta)}\ud t.$$
Since the metric $g$ is complete,  we have  $d_g(0,r\theta)\to\infty$ as $r\to\infty$.
Based on this estimate, we obtain that 
	$$\fint_{\ms^{n-1}}\int^{r}_0e^{u(t\theta)}\ud t\ud\sigma\to \infty, \quad \mathrm{as}\; r\to\infty.$$
With the help of Fubini's theorem and Lemma \ref{CQY lemma}, one has
$$\fint_{\ms^{n-1}}\int^{r}_0e^{u(t\theta)}\ud t\ud\sigma=\int^r_0e^{\bar u(t)+o(1)}\ud t$$
which yields that $d_{\bar g}(0,r\theta)\to \infty$ as $r\to\infty.$ Thus, $\bar g$ is also complete.
Noticing that 
\begin{equation*}
	(-\Delta)^{\frac{n}{2}}\bar u= \overline{(-\Delta)^{\frac{n}{2}} u}.
\end{equation*} Thus there holds
	$$\int_{\mr^n}(-\Delta)^{\frac{n}{2}}\bar u\ud x= \int_{\mr^n}(-\Delta)^{\frac{n}{2}} u\ud x$$
	and
	$$\int_{\mr^n}|(-\Delta)^{\frac{n}{2}}\bar u|\ud x\leq \int_{\mr^n}|(-\Delta)^{\frac{n}{2}}u|\ud x<+\infty.$$
Set the logarithmic potential
$$\mathcal{L}(x):=\frac{2}{(n-1)!|\mathbb{S}^n|}\int_{\mr^n}\log\frac{|y|}{|x-y|}(-\Delta)^{\frac{n}{2}}\bar u(|y|)\ud y.$$
It is easy to check that $\mathcal{L}(x)$ is also radially symmetric and then $\mathcal{L}=\bar{\mathcal{L}}$.
Following the argument  the Lemma \ref{lem:bar u}, one has $|\bar{\mathcal{L}}(r)|\leq C\log (r+2)$ and then $|\bar u(r)-\bar{\mathcal{L}}(r)|\leq C\log (r+2)$.
Using \eqref{Green function}, one has
$(-\Delta)^{\frac{n}{2}}(\bar u-\bar{\mathcal{L}})=0.$  Using the Liouville theorem for polyharmonic function (See Lemma 2.12 in \cite{Li Adv}), one has
$\bar u-\bar{\mathcal{L}}\equiv C.$ Then, we show that $\bar u$ is also normal.

This completes the proof of Lemma \ref {lem: bar g}.
\end{proof}

\section{$Q$-curvature and sectional curvature}\label{sec:Q and sec}

This section is devoted to establishing the relationship between the top 
$Q$-curvature and the sectional curvature of smooth complete normal metrics on $\mr^n$.  We recall that, with many people's great efforts, the  following Cohn-Vossen type inequality has been established, for example, through works \cite{CV}, \cite{Huber 57}, \cite{CQY Duke},  \cite{Fang},  and \cite{NX} and the reference therein.  For later use, we record this as a lemma.
\begin{lemma}\label{lem:Cohn-Vossen}
		For a smooth complete normal metric $g=e^{2u}|dx|^2$   on $\mr^n$ with 
	$n\geq 2$, there holds
	$$\int_{\mr^n}Q_g^{(n)}\ud\mu_g\leq \frac{(n-1)!|\mathbb{S}^n|}{2}.$$
\end{lemma}

\begin{proof}
With the help of Lemma \ref{lem: bar g}, we have
$$\int^r_0e^{\bar u(s)}\ud s\to \infty,\quad  \mathrm{as}\; r\to\infty.$$
Combine with Lemma \ref{lem:bar u} to yield:
$$\alpha_0\leq 1.$$
Thus the proof is complete.
\end{proof}

We now state and prove a result which builds up a connection  between the top $Q$-curvature and sectional curvature of the metric $g$. We should point out that our Theorem \ref{thm:positive sectional for normal metric} follows from this and the previous lemma.

\begin{theorem}\label{thm:positive sec for alpha_0 leq2}
Given  a smooth   normal metric $g=e^{2u}|dx|^2$   on $\mr^n$ where 
$n\geq 2$ and $\alpha_0\leq 2$. Then the non-negativity (respectively non-positivity) of the top $Q$-curvature everywhere on $\mr^n$ determines the non-negativity  (respectively non-positivity) of the sectional curvature everywhere in the same space.
\end{theorem}

\begin{proof}
	
We first notice that if $n=2$, the statements hold automatically. We only need to deal with $n\geq 3$.
For the conformal metric $g=e^{2u}|dx|^2$ on $\mr^n$, let $\vec{e}_1$ and $\vec{e}_2$ be  orthonormal unit vectors with respect to the Euclidean metric. Then the sectional curvature $K_g(\vec{e}_1,\vec{e}_2)$ with respect to $g$  of the linear subspace spanned by $\vec{e}_1$ and $\vec{e}_2$ satisfies
	\begin{eqnarray}\label{sec-equ}
		e^{2u}K_g(\vec{e}_1,\vec{e}_2) & =& - \nabla^2 u(\vec{e}_1,\vec{e}_1)-\nabla^2u(\vec{e}_2, \vec{e}_2) \nonumber\\
		& & +\langle \nabla u, \vec{e}_1\rangle^2+\langle \nabla u, \vec{e}_2\rangle^2-|\nabla u|^2.
	\end{eqnarray}
More details about conformal change can be found in Page 58 of the wonderful book \cite{Besse}.
Since the conformal factor $u$ satisfies the identity \eqref{normal solution}, a direct computation yields that 
	$$u_{i}(x)=-\int_{\mr^n}\frac{(x_i-y_i)}{|x-y|^2}\ud\nu(y)$$
	and 
	\begin{equation}\label{u_ij-repren}
		u_{ij}(x)=-\int_{\mr^n}\frac{\delta_{ij}}{|x-y|^2}\ud\nu(y)+2\int_{\mr^n}\frac{(x_i-y_i)(x_j-y_j)}{|x-y|^4}\ud\nu(y).
	\end{equation}

Up to a rotation of the coordinates, it is harmless to choose $\vec{e}_1=(1,0, \cdots, 0)$ and  $\vec{e}_2=(0,1, \cdots, 0)$.
Inserting these into the equation \eqref{sec-equ}, one has
\begin{align*}
	&e^{2u}K_g(\vec{e}_1,\vec{e}_2)\\
	=&-u_{11}-u_{22}-\sum^n_{i=3}u_i^2\\
	=&\int_{\mr^n}\frac{2}{|x-y|^2}\ud\nu(y)-2\int_{\mr^n}\frac{(x_1-y_1)^2+(x_2-y_2)^2}{|x-y|^4}\ud\nu(y)-\sum^n_{i=3}u_i^2\\
	=&\int_{\mr^n}\frac{2\sum^n_{i=3}(x_i-y_i)^2}{|x-y|^4}\ud\nu(y)-\sum^n_{i=3}u_i^2.
\end{align*}
This if $Q_g^{(n)}\leq 0$, then the measure $\ud\nu(y) \leq 0$. Thus it  is easy  to see that 
$$K_g(\vec{e}_1,\vec{e}_2)\leq 0.$$

On the other hand, if  $Q^{(n)}_g\geq 0$ everywhere on $\mr^n$, we clearly have the measure $\ud\nu(y) \geq 0$. Hence, with the help of H\"older's inequality, one has
$$\int_{\mr^n}\ud\nu(y)\int_{\mr^n}\frac{\sum^n_{i=3}(x_i-y_i)^2}{|x-y|^4}\ud\nu(y)\geq \sum^n_{i=3}u_i^2.$$
This estimate, together with the assumption $\alpha_0\leq 2$, immediately implies
$$e^{2u}K_g(\vec{e}_1,\vec{e}_2)\geq (2-\alpha_0)\int_{\mr^n}\frac{\sum^n_{i=3}(x_i-y_i)^2}{|x-y|^4}\ud\nu(y)\geq 0.$$
The proof of Theorem \ref{thm:positive sec for alpha_0 leq2} is complete. 
\end{proof}

\vspace{3em}

Motivated by our earlier work joint with Wei, Theorem 4.5 in \cite{LWX}, by employing  the integral representation of conformal factor under the assumption that the $Q^{(2k)}_g$-curvature has a slow decay barrier near infinity (See Theorem 2.6 in \cite{LX JFA}), we can, in fact, bound the sectional curvature from below in term of its scalar curvature.

To this end, let us consider a conformal metric $g=e^{2u}|dx|^2$ on $\mr^n$. For each integer $1\leq k<\frac{n}{2}$, $2k$-th order $Q$-curvature $Q^{(2k)}_g$ satisfies
\begin{equation}\label{Q_2k_def}
	(-\Delta)^ke^{\frac{n-2k}{2}u}=Q^{(2k)}_ge^{\frac{n+2k}{2}u}.
\end{equation}
Same as in \cite{LX JFA}, we say $Q^{(2k)}_g$ has slow decay barrier near infinity if,  for $|x|\gg1$ and $-2k<s\leq 0$, one has
$$Q_g^{(2k)}\geq C|x|^s.$$

Then we can state our findings as follows:
\begin{theorem}
	Let us fix two integers $n>2k$ and $k\geq 2$. Suppose  $g=e^{2u}|dx|^2$ on $\mr^n$ is a smooth complete conformal metric. Suppose $Q^{(2k)}_g$-curvature is non-negative and has a  slow decay barrier near infinity. Then the sectional curvature of the metric satisfies the lower bound estimate in terms of scalar curvature:
	$$sec_g\geq \frac{2k-n}{4(k-1)(n-1)}R_g.$$
\end{theorem}

\begin{proof}
	We first define the function $v$ by $v(x):=e^{\frac{n-2k}{2}u(x)}$.  In terms of $v$, the equation \eqref{Q_2k_def} can be written as 
	$$(-\Delta)^kv=Q^{(2k)}_gv^{\frac{n+2k}{n-2k}}.$$
	Also in terms of $v$,  the scalar curvature of $g$ can be calculated by
	\begin{equation}\label{R_g repre}
	R_g=-\frac{4(n-1)}{n-2k}v^{-\frac{4}{n-2k}}\left(\frac{\Delta v}{v}+\frac{2(k-1)}{n-2k}\frac{|\nabla v|^2}{v^2}\right).
\end{equation}
With the help of Theorem 2.6 in \cite{LX JFA}, we have the following integral representation of $v(x)$ as follows 
	\begin{equation}\label{equ:integeral-for-v}
		v(x)=\frac{1}{C(n,k)}\int_{\mr^n}\frac{Q^{(2k)}_g(y)v(y)^{\frac{n+2k}{n-2k}}}{|x-y|^{n-2k}}\ud y.
	\end{equation}
Thus, for any $1\leq i\leq 2k-1$,  Lemma 2.4 in \cite{LX JFA} or a direct computation yields that 
	\begin{equation}\label{equ:nabla-for-v}
		\nabla^i v=\frac{1}{C(n,k)}\int_{\mr^n}\nabla^i_x|x-y|^{2k-n}Q^{(2k)}_g(y)v(y)^{\frac{n+2k}{n-2k}}\ud y
	\end{equation}
	where $C(n,k)$ is a positive constant depending on $n,k$.
	To simplify the formula, we denote by $\ud\mu(y)$ the integral measure, that is, 
	$$\ud\mu(y):=\frac{1}{C(n,k)}Q^{(2k)}_g(y)v(y)^{\frac{n+2k}{n-2k}}\ud y.$$
	 By the relation between $v$ and $u$, a straightforward calculation gives
	\begin{equation*}
		u_{i}=\frac{2}{n-2k}\frac{v_i}{v}
	\end{equation*}
and 
	\begin{equation}\label{u_ij}
		u_{ij}=\frac{2}{n-2k}\left(\frac{v_{ij}}{v}-\frac{v_iv_j}{v^2}\right).
	\end{equation}
In particular, by taking the trace of \eqref{u_ij}, one has
	\begin{equation}\label{Delta-u}
		\Delta u=\frac{2}{n-2k}\left(\frac{\Delta v}{v}-\frac{|\nabla v|^2}{v^2}\right).   
	\end{equation}
However, the equations \eqref{equ:integeral-for-v} and \eqref{equ:nabla-for-v} can be applied to do the following computations:
\begin{equation}\label{v_i repre}
	v_{i}=(2k-n)\int_{\mathbb{R}^n}\frac{(x_i-y_i)}{|x-y|^{n-2k+2}}\ud\mu(y)
\end{equation}
	and 
	$$v_{ij}=(2k-n)\int_{\mathbb{R}^n}\frac{\delta_{ij}}{|x-y|^{n-2k+2}}\ud\mu(y)+(2k-n)(2k-2-n)\int_{\mathbb{R}^n}\frac{(x_i-y_i)(x_j-y_j)}{|x-y|^{n-2k+4}} \ud\mu(y).$$
	Immediately, there holds 
	\begin{equation}\label{Delta v repre}
		\Delta v=(2k-n)(2k-2)\int_{\mathbb{R}^n}\frac{1}{|x-y|^{n-2k+2}}\ud\mu(y).
	\end{equation}
		As in the proof of previous theorem, without loss of generality, one may choose  $\vec{e}_1=(1,0, \cdots, 0)$ and  $\vec{e}_2=(0,1, \cdots, 0)$ and 	insert these vectors into the equation \eqref{sec-equ}. Then,  there holds
	\begin{align*}
			&e^{2u}K_g(\vec{e}_1,\vec{e}_2)\\
		=&-u_{11}-u_{22}+u_1^2+u_2^2-|\nabla u|^2\\
		=&-\frac{2}{n-2k}\left(\frac{v_{11}+v_{22}}{v}-\frac{v_1^2+v_2^2}{v^2}\right)+\frac{4}{(n-2k)^2}\frac{v_1^2+v_2^2}{v^2}-\frac{4}{(n-2k)^2}\frac{|\nabla v|^2}{v^2}\\
		=&-\frac{2}{n-2k}\frac{v_{11}+v_{22}}{v}-\frac{2(n-2k+2)}{(n-2k)^2}\frac{\sum^n_{i=3}v_i^2}{v^2}+\frac{2}{n-2k}\frac{|\nabla v|^2}{v^2}\\
		=&\frac{4}{v}\int_{\mr^n}\frac{\ud\mu(y)}{|x-y|^{n-2k+2}}+\frac{2(2k-2-n)}{v}\int_{\mr^n}\frac{|x_1-y_1|^2+|x_2-y_2|^2}{|x-y|^{n-2k+4}}\ud\mu(y)\\
		&-\frac{2(n-2k+2)}{(n-2k)^2}\frac{\sum^n_{i=3}v_i^2}{v^2}+\frac{2}{n-2k}\frac{|\nabla v|^2}{v^2}\\
		=&\frac{2(2k-n)}{v}\int_{\mr^n}\frac{\ud\mu(y)}{|x-y|^{n-2k+2}}+\frac{2}{n-2k}\frac{|\nabla v|^2}{v^2}\\
		&+\frac{2(n-2k+2)}{v}\int_{\mr^n}\frac{\sum^n_{i=3}|x_i-y_i|^2}{|x-y|^{n-2k+4}}\ud \mu(y)-\frac{2(n-2k+2)}{(n-2k)^2}\frac{\sum^n_{i=3}v_i^2}{v^2}.
	\end{align*}
	H\"older's  inequality, with the help of the equations \eqref{equ:integeral-for-v} and \eqref{v_i repre}, implies
	$$v_i^2\leq (n-2k)^2v\cdot\int_{\mr^n}\frac{|x_i-y_i|^2}{|x-y|^{n-2k+4}}\ud\mu(y).$$
	This inequality, together with the equation \eqref{Delta v repre}, shows that
	\begin{equation}\label{sec lower bound}
		e^{2u}K_g(\vec{e}_1,\vec{e}_2)\geq \frac{1}{k-1}\frac{\Delta v}{v}+\frac{2}{n-2k}\frac{|\nabla v|^2}{v^2}.
	\end{equation}
Compare this with the identity \eqref{R_g repre} to  achieve:
$$K_g(\vec{e}_1,\vec{e}_2)\geq\frac{2k-n}{4(n-1)(k-1)}R_g,$$ by observing that $ (v^{-\frac{4}{n-2k}}) = e^{-2u}$.

Thus, the proof is complete.
\end{proof}

\section{Volume ratio and $Q$-curvature}\label{sec:volume ratio}

Throughout this section, we aim to prove a key lemma concerning the connection between the volume ratio and the integral of $Q$-curvature. For brevity, we denote by $B_r^g(x_0)$ the geodesic ball of radius $r$ centered at $x_0$ with respect to the metric $g$, and by $V_g(\Omega)$ the volume of $\Omega$ with respect to $g$.

In the  previous work \cite{Li Adv}, the first  author introduced a volume entropy defined as
\[
\tau(g):=\limsup_{r\to\infty}\frac{\log V_g(B_r(0))}{\log |B_r(0)|}
\]
and gave its precise value when the conformal metric is complete and normal. Theorem 1.1 in \cite{Li Adv} showed that for a complete normal metric \eqref{normal solution} with finite total $Q$-curvature,
\[
\tau(g)=1-\alpha_0.
\]
Moreover, that result was used to characterize the Euclidean ball of radius $r$ with respect to two different metrics. In this section, we characterize the volume growth of the geodesic ball, which is a more intrinsic quantity.

Using the well-known Bishop-Gromov volume comparison theorem, the volume quotient
\begin{equation}\label{volume_quotient}
	\frac{V_g(B_r^g(x_0))}{|\mathbb{B}^n| r^n}
\end{equation}
is monotone non-increasing if the Ricci curvature of the manifold is non-negative. As a direct corollary, the limit of this quotient as $r\to\infty$ exists in this situation.  Without the sign of curvature, for a complete normal metric $g=e^{2u}|dx|^2$ on $\mathbb{R}^2$, Hartman \cite{Harman} and Shiohama \cite{Shiohama} showed that
\begin{equation}\label{Hartman-Shiohama}
	\lim_{r\to\infty}\frac{V_g(B_r^g(0))}{\pi r^2}=1-\frac{1}{2\pi}\int_{\mathbb{R}^2} K_g \, d\mu_g.
\end{equation}

Recently, S. Ma \cite{Ma}  considered $g=e^{2u}|dx|^2$ on $\mr^n$ with non-negative Ricci curvature  and used a refined singularity estimate for non-negative $n$-superharmonic functions to give a limit of volume quotient  \eqref{volume_quotient} related to  the asymptotic behavior of $u$  near infinity based on his joint  work with J. Qing \cite{Ma-Qing}.

Inspired by these results, we shall study the volume quotient \eqref{volume_quotient} for complete normal metrics. As a main result of this section, we state and prove the following theorem.

\begin{theorem}\label{lem:volume ratio and Q-curvature}
	For a smooth complete normal metric $g=e^{2u}|dx|^2$  on $\mr^n$ with $n\geq 2$,  there holds
	\begin{equation}\label{volume ratio limit}
		\liminf_{r\to\infty}\frac{V_g(B_r^g(0))}{|\mathbb{B}^n|r^n}\geq(1-\alpha_0)^{n-1}.
	\end{equation}
\end{theorem}

\begin{remark}
We tend to believe that the inequality in the above statement is indeed an equality just  like the identity \eqref{Hartman-Shiohama}.  We will discuss this after we provide the proof of Theorem \ref{thm:main theorem} in Section \ref{sec:proof of main theorem}.
\end{remark}

The proof  of Theorem \ref{lem:volume ratio and Q-curvature} should not be provided until the end of this section.

Motivated by S. Ma's Theorem 3.1 in \cite{Ma}, and our estimate on the spherical average Lemma \ref{lem: bar g}, we first treat a radially symmetric version of Lemma \ref{lem:volume ratio and Q-curvature}.

\begin{lemma}\label{lem: Vbarg}
For  a smooth complete, conformal normal metric $g=e^{2u}|dx|^2$ on $\mr^n$,  let us consider its radially symmetric metric $\bar g=e^{2\bar u}|dx|^2$. For this rotationally symmetric metric, there holds
	\[\lim_{r\to\infty}\frac{V_{\bar g}(B_r^{\bar g}(0))}{|\mathbb{B}^n|r^{n}}= (1-\alpha_0)^{n-1}.\]
\end{lemma}
\begin{proof}
	Let us first consider a new radial function $f$ defined by
	$$f(r):=\int_0^re^{\bar u(s)}\ud s.$$
Our Lemma \ref{lem: bar g} shows that  $f(r)\to \infty$ as $r\to\infty$.
Therefore, we can calculate the volume of ball $B_{f(r)}^{\bar g}(0)$ by the formula:
$$V_{\bar g}(B_{f(r)}^{\bar g}(0))=\int^r_0|\mathbb{S}^{n-1}|s^{n-1}e^{n\bar u(s)}\ud s.$$
	Applying  L'H\^opital's rule and the fact that $f'(r)=e^{\bar u(r)}$, we reach at:
$$
	\lim_{r\to\infty}	\frac{V_{\bar g}(B_{f(r)}^{\bar g}(0))}{|\mathbb{B}^n|(f(r))^n}=\lim_{r\to\infty}\frac{|\mathbb{S}^{n-1}|r^{n-1}e^{n\bar u(r)}}{|\mathbb{B}^n|n(f(r))^{n-1}e^{\bar u(r)}}
	=\lim_{r\to\infty}\left(\frac{re^{\bar u(r)}}{f(r)}\right)^{n-1}
$$
Then, the L'H\^opital's rule once again can be used to value the limit:
	\[\lim_{r\to\infty}\frac{re^{\bar u(r)}}{f(r)}=\lim_{r\to\infty}\frac{e^{\bar u(r)}+re^{\bar u(r)}\bar u'(r)}{e^{\bar u(r)}}=\lim_{r\to\infty}(1+r\cdot\bar u'(r))=1-\alpha_0\]
	where the last equality comes from Lemma \ref{lem:ru'}.
Thus, we finish the proof.	
\end{proof}

Thus the key ingredient in proof of our theorem \ref{lem:volume ratio and Q-curvature} is to compare the geodesic ball volume growth as well as the geodesic distance of the two conformal metrics $g$ and $\bar g$ under the condition $\alpha_0 < 1$ and to see if they are near each other.  To this end, the main tool is the so-called strong $A_\infty$ weight, introduced by David and Semmes \cite{DS}.  For readers' convenience, we provide some background materials about it. Again let us consider a smooth conformal metric $g=e^{2u}|dx|^2$ on $\mr^n$. The conformal  volume factor $e^{nu}$ is said to be an  $A_\infty$ weight if, for all  Euclidean balls  $B_r(x_0)$, there holds
$$\fint_{B_r(x_0)}e^{nu}\ud x \leq C\exp\left(\fint_{B_r(x_0)}nu\ud x\right)$$
where $C$ is a positive constant independent of $r$ and $x_0$.
An important property of $A_\infty$ weight is concerning the comparison of geodesic distance $d_g(x,y)$ and the measure distance $\delta_g(x,y)$:
\begin{equation}\label{A_infty weight property}
	d_g(x,y)\leq C\delta_g(x,y)
\end{equation}
where $C$ is independent of $x,y$ and the measure distance  $\delta_g(x,y)$ is given by
$$\delta_g(x,y):=\left(\int_{B_{\frac{|x-y|}{2}}(\frac{x+y}{2})}e^{nu(z)}\ud z\right)^{\frac{1}{n}}.$$
An $A_\infty$ weight is called strong $A_\infty$ weight if the reversed version of the estimate \eqref{A_infty weight property} also holds true, that is, 
\begin{equation}\label{strong A_infty weight property}
	\delta_g(x,y)\leq Cd_g(x,y).
\end{equation}

A breakthrough concerning the geometric condition under which $e^{nu}$
is a strong $A_\infty$ weight is due to \cite{BHS}, who established the result under the assumption that the $L^1$-norm of the $n$-th order 
$Q$-curvature is sufficiently small. This was later extended by Wang \cite{Wang IMRN} to the optimal case. For the readers' convenience, we state this as a lemma.
\begin{lemma}\label{lem:strong A_infty} (Theorem 1.3 in \cite{BHS}, Corollary 1.7 in \cite{Wang IMRN})
	Let  $g=e^{2u}|\ud x|^2$ be  a smooth complete normal metric on $\mr^n$ with $n\geq 2$ and  $\alpha_0<1$.  Then, $e^{nu}$ is a strong $A_\infty$ weight.
\end{lemma}
\begin{remark}
Wang proved this in the four-dimensional case in \cite{Wang IMRN} and remarked that the method also works for all even-dimensional cases (see Corollary 1.7 in \cite{Wang IMRN}). In fact, under the normal metric assumption, the odd-dimensional case is also valid, since the quasiconformal flow method developed in \cite{BHS} applies for every 
$n\geq 2$ (see Theorem 1.2 therein).
\end{remark}

We observe that the property of strong $A_\infty$ weight can be used to control the volume growth of geodesic balls.
\begin{lemma}\label{lem: V_barg and V_g growth}
Let  $g=e^{2u}|\ud x|^2$ be  a smooth complete  normal metric on $\mr^n$ with $n\geq 2$ with the property that $\alpha_0<1$.  Then for each  $r>0$, the following statements  hold:
\begin{equation}\label{V_g}
	C^{-1}r^n\leq V_g(B_r^g(0))\leq Cr^n,
\end{equation}
and 
\begin{equation}\label{V_bar g}
	C^{-1}r^n\leq V_{\bar g}(B_r^{\bar g}(0))\leq Cr^n.
\end{equation}
Meanwhile, for $r>0$ and any $z_r\in\partial B_r^g(0)$, one has
\begin{equation}\label{Vz_r}
	C^{-1}r^n\leq \int_{B_{|z_r|}(0)}e^{nu}\ud x\leq Cr^n.
\end{equation}
and 
\begin{equation}\label{d_bar g equiv r}
	C^{-1}r\leq d_{\bar g}(0,z_r)\leq Cr.
\end{equation}
\end{lemma}

\begin{proof}
In fact, the first estimate \eqref{V_g} has been established by Y. Sire and Y. Wang, see Proposition 4.1 in \cite{Sire-Wang}.  For readers' convenience, we will sketch the proof again.
	
Under  the assumption $\alpha_0<1$, Lemma \ref{lem:strong A_infty} shows that $e^{nu}$ is a strong $A_\infty$ weight which yields that  the geodesic  distance and measure distance are equivalent in the sense that
	 \begin{equation}\label{strong_A_infty weight}
	 	C^{-1}\delta_g(x,y)\leq d_g(x,y)\leq C\delta_g(x,y).
	 \end{equation}
Due to $B_{|z_r|}(0)\subset B_{2|z_r|}(\frac{z_r}{2})$,  by the doubling property of strong $A_\infty$ weight (See \cite{Semmes}) and the equivalence between two distances \eqref{strong_A_infty weight}, there holds
\begin{align*}
	\int_{B_{|z_r|}(0)}e^{nu}\ud x\leq& \int_{ B_{2|z_r|}(\frac{z_r}{2})}e^{nu}\ud x\\
	\leq &C\int_{ B_{\frac{|z_r|}{2}}(\frac{z_r}{2})}e^{nu}\ud x\\
	=&C\delta_g(0,z_r)^n\\
	\leq &Cd_g(0,z_r)^n=Cr^n.
\end{align*}

Meanwhile,  using the estimate \eqref{strong_A_infty weight}, it is not hard to see that
$$\int_{B_{|z_r|}(0)}e^{nu}\ud x\geq \int_{ B_{\frac{|z_r|}{2}}(\frac{z_r}{2})}e^{nu}\ud x=\delta_g(0,z_r)^n\geq C^{-1}d_g(0,z_r)^n=C^{-1}r^n.$$
Thus, we show that the desired estimate \eqref{Vz_r} holds true. 

With the estimate \eqref{Vz_r} at hand, the estimate \eqref{V_g} can be easily obtained by choosing $z_r^{(1)}, z_r^{(2)}\in \partial B_r^g(0)$ such that $B_{|z_r^{(1)}|}(0)\subset B_r^g(0)\subset B_{|z_r^{(2)}|}(0)$.

With the help of Lemma \ref{lem: bar g}, the same argument also works for $\bar g=e^{2\bar u}|\ud x|^2$,  and therefore the estimate \eqref{V_bar g} holds.

Finally, we are going to deal with the estimate \eqref{d_bar g equiv r}.
One one hand, there holds
\begin{align*}
	d_{\bar g}(0,z_r)^n\leq &C\delta_{\bar g}(0,z_r)^n\\
	\leq &C\int_{B_{|z_r|}(0)}e^{n\bar u}\ud x\\
	\leq &C\int_{B_{|z_r|}(0)}e^{n u}\ud x\leq Cr^n.
\end{align*}

On the other hand, using the doubling property of the strong $A_\infty$ weight and choosing $k=n$ in Lemma~\ref{lem:e^kw}, we have
\begin{align*}
	d_{\bar g}(0,z_r)^n\geq &C\delta_{\bar g}(0,z_r)^n\\
	\geq &C\int_{B_{2|z_r|}(\frac{z_r}{2})}e^{n\bar u}\ud x\\
	\geq &C\int_{B_{|z_r|}(0)}e^{n \bar u}\ud x\\
	\geq &C\int_{B_{|z_r|}(0)}e^{nu}\ud x\geq Cr^n.
\end{align*}
Hence our estimate \eqref{d_bar g equiv r} follows by combining these two estimates. Therefore our proof is complete.
\end{proof}

In fact, the distance function $d_g$ of metric $g$ is not suitable for radially symmetric computations. Therefore, we introduce the following two quantities. For polar coordinates $x = r\theta$, define
\[
l(x) := \int_0^r e^{u(t\theta)} \, dt, \qquad \bar l(x) := \int_0^r e^{\bar u(t)} \, dt.
\]
Since $\bar g=e^{2\bar u}|dx|^2$ is radially symmetric, $\bar l(x)$  indeed represents  the geodesic distance from $x$ and the origin.

The following lemma shows that the two distance-type functions $l(x)$ and $\bar l(x)$ are almost the same in the volume measure.
\begin{lemma}\label{lem: |l(x)-bar l(x)|}
Consider a complete normal metric $g=e^{2u}|\ud x|^2$  on $\mr^n$ with  $n\geq 2$ and   its radially symmetric counterpart  $\bar g=e^{2\bar u}|dx|^2$. Then, there holds
	$$\lim_{r\to\infty}\frac{\int_{B_r^{\bar g}(0)}|l(x)-\bar l(x)|\ud\mu_{\bar g}}{r\cdot V_{\bar{g}}(B_r^{\bar g}(0))}=0.$$
\end{lemma}
\begin{proof}

For any $r>0$, choose $y_r\in \mr^n$ such that $B_{|y_r|}(0)=B_r^{\bar g}(0).$ Since $\bar g$ is radially symmetric, it is easy to obtain that $\bar l(y_r)=r.$

The elementary inequality $e^t\geq 1+t$ can be applied to get the following estimate, noticing that $ w = u - \bar u$,
$$|e^{w}-1|\leq |e^w-1-w|+|w|=e^{w}-1-w+|w|.$$
This estimate, with Lemmas \ref{CQY lemma} and \ref{lem:|w|},  shows that
\begin{equation}\label{|e^w-1| to 0}
	\lim_{t\to\infty}\int_{\partial B_1(0)}|e^{w(t\theta)}-1|\ud\sigma(\theta)=0.
\end{equation}
	With the help of Fubini's theorem, there holds
\begin{align*}
	& J(r) := \int_{B_r^{\bar g}(0)}|l(x)-\bar l(x)|\ud\mu_{\bar g}\\
	=&\int^{|y_r|}_0s^{n-1}e^{n\bar u(s)}\int_{\partial B_1(0)}|l(s\theta)-\bar l(s\theta)|\ud\sigma(\theta)\ud s\\
	\leq &\int^{|y_r|}_0s^{n-1}e^{n\bar u(s)}\int_{\partial B_1(0)}\int^s_0|e^{w(t\theta)}-1|e^{\bar u(t)}\ud t\ud\sigma(\theta)\ud s\\
	=&\int^{|y_r|}_0s^{n-1}e^{n\bar u(s)}\int^s_0e^{\bar u(t)}\int_{\partial B_1(0)}|e^{w(t\theta)}-1|\ud\sigma(\theta)\ud t\ud s.
	\end{align*}
	and 
	$$ r \cdot V_{\bar{g}} (B_r^{\bar g}(0)) = \int^{|y_r|}_0e^{\bar u(s)}\ud s\int^{|y_r|}_0 s^{n-1} e^{n\bar u(s)} |\mathbb{S}^{n-1}| \ud s. $$
Then, apply L'Hospital's rule twice to get:
\begin{align*}
& 0 \leq \lim_{r\to\infty}\frac{\int_{B_r^{\bar g}(0)}|l(x)-\bar l(x)|\ud\mu_{\bar g}}{r\cdot V_{\bar{g}}(B_r^{\bar g}(0))}\\
\leq & \lim_{|y_r| \to\infty} \frac{ |y_r|^{n-1}e^{n\bar u(|y_r|)}\int_0^{|y_r|} e^{{\bar u}(t)}\int_{\partial B_1(0)}|e^{w(t\theta)}-1|\ud\sigma(\theta) \ud t} { e^{{\bar u}(|y_r|)} \int_0^{|y_r|} t^{n-1} e^{n {\bar u}(t)} dt |\mathbb{S}^{n-1}| + |\mathbb{S}^{n-1}|  |y_r|^{n-1}e^{n\bar u(|y_r|)}\int_0^{|y_r|} e^{{\bar u}(t)} dt}\\
\leq & \lim_{|y_r| \to\infty}   \frac{ \int_0^{|y_r|}  e^{{\bar u}(t)} \int_{\partial B_1(0)} |e^{w(t\theta)}-1|\ud\sigma(\theta) \ud t}  { |\mathbb{S}^{n-1}| \int_0^{|y_r|} e^{{\bar u}(t)} dt}\\
= &  \frac{1}{|\mathbb{S}^{n-1}|}\lim_{|y_r| \to\infty}  \int_{\partial B_1(0)}|e^{w(|y_r| \theta)}-1|\ud\sigma(\theta)  =  0
\end{align*}
by the limit \eqref{|e^w-1| to 0}.  The proof is complete.		
\end{proof}

We now are ready to compare the volume of the metric $g$ and the one for its spherical average ${\bar g}$.

\begin{lemma}\label{lem:V_g Omega_r vs V_ bar g}
	For a smooth complete, conformal normal metric $g=e^{2u}|dx|^2$  on $\mr^n$ with  $n\geq 2$ and $\alpha_0<1$,  it holds true that 
	$$\lim_{r\to\infty}\frac{V_g(B_r^g(0))}{V_{\bar g}(B_r^g(0))}=1.$$
\end{lemma}
\begin{proof}

For any $r>0$, choose $e_r\in \partial B_r^g(0)$ such that $B_r^g(0)\subset B_{|e_r|}(0)$. Since the metric is complete, we have $|e_r|\to\infty$ as $r\to\infty$.
By the same argument as in the proof for the limit \eqref{|e^w-1| to 0}, one also has
\begin{equation*}
	\fint_{\partial B_t(0)}|e^{nw}-1|\ud\sigma\to 0, \quad \mathrm{as}\;\; t\to\infty.
\end{equation*}
With the help of this estimate, together with  the fact that the volume of the metric ${\bar g}$ is infinity (the limit \eqref{infinte V_bar g}) and  L'H\^opital's rule, we obtain that
\begin{equation}\label{|e^nw-1|to0}
	\lim_{t\to\infty}\frac{\int_{B_t(0)}|e^{nw}-1|\ud\mu_{\bar g}}{\int_{B_t(0)}\ud\mu_{\bar g}}=\lim_{t\to\infty}\fint_{\partial B_t(0)}|e^{nw}-1|\ud\sigma=0.
\end{equation}
Then,  apply  the estimate \eqref{|e^nw-1|to0} to conclude that
\begin{align*}
&|V_g(B_r^g(0))-V_{\bar g}(B_r^g(0))|\\
	\leq &\int_{B_{|e_r|}(0)}|e^{nw}-1|\ud\mu_{\bar g}\\
=&\frac{\int_{B_{|e_r|}(0)}|e^{nw}-1|\ud\mu_{\bar g}}{\int_{B_{|e_r|}(0)}\ud\mu_{\bar g}}\int_{B_{|e_r|}(0)}e^{n\bar u}\ud x\\
=&\frac{\int_{B_{|e_r|}(0)}|e^{nw}-1|\ud\mu_{\bar g}}{\int_{B_{|e_r|}(0)}\ud\mu_{\bar g}} \cdot \frac{\int_{B_{|e_r|}(0)}e^{n\bar u}\ud x}{\int_{B_{|e_r|}(0)}e^{n u}\ud x}\cdot \int_{B_{|e_r|}(0)}e^{n u}\ud x\\
=& o(1)V_{g}(B_r^g(0)),
\end{align*}  by Jensen's inequality, $\frac{\int_{B_{|e_r|}(0)}e^{n\bar u}\ud x}{\int_{B_{|e_r|}(0)}e^{n u}\ud x} \leq 1$ and the fact that $ \int_{B_{|e_r|}(0)}e^{n u}\ud x \leq C r^n \leq C V_{g}(B_r^g(0))$ due to the estimate (\ref{V_g}). Then both sides divides by $V_{g}(B_r^g(0))$ to see that 
$$ \lim_{r\to\infty} \frac{V_{\bar g} (B_r^g(0))}{V_g(B_r^g(0))}=1,$$ which is equivalent to our claim.
Thus,  the proof of Lemma \ref{lem:V_g Omega_r vs V_ bar g} is complete.
\end{proof}

Finally, we are ready to prove Theorem \ref{lem:volume ratio and Q-curvature}. 

\vskip .2in

\noindent {\bf Proof of Theorem \ref{lem:volume ratio and Q-curvature}:}
 Recall that our main claim of Theorem \ref{lem:volume ratio and Q-curvature} is the following inequality:
	\begin{equation*}
		\liminf_{r\to\infty}\frac{V_g(B_r^g(0))}{|\mathbb{B}^n|r^n}\geq(1-\alpha_0)^{n-1}.
	\end{equation*}

First of all, Lemma \ref{lem:Cohn-Vossen} is just said that $\alpha_0\leq 1$. 
Since the case $\alpha_0 = 1$ is trivial,  we just need to prove the inequality under the assumption that $\alpha_0<1$.
In order to do this, let us consider the  set
$$\Omega_s:=\{x|x\in\mr^n, l(x)<s\}.$$
By the definition, for any $x\in \Omega_s\setminus \{0\}$, one can easily see that
$$d_g(0,x)\leq \int^{|x|}_0e^{u(t\frac{x}{|x|})}\ud t =l(x)<s$$
which yields that 
\begin{equation}\label{Omega_s subset B_s}
	\Omega_s\subset B_s^g(0).
\end{equation}
Pick a constant number $\delta>0$ and  consider the dilated set
$$A_{\delta, r}:=\{x\in B_r^{\bar g}(0)\mid |l(x)-\bar l(x)|<\delta r\}.$$
Then, a direct consequence of this set is, for a $x\in A_{\delta, r}$, 
$$l(x)\leq \bar l(x)+\delta r<(1+\delta)r.$$
which shows that
\begin{equation}\label{Adelta subet Omega 1+delta r}
	A_{\delta, r}\subset \Omega_{(1+\delta)r}.
\end{equation}
Now it follows from the definition of $A_{\delta, r}$ that
$$\frac{V_{\bar g}(B_r^{\bar g}(0)\setminus A_{\delta, r})}{V_{\bar g}(B_r^{\bar g}(0))}\leq \frac{1}{\delta r}\frac{\int_{B_r^{\bar g}(0)}|l(x)-\bar l(x)|\ud\mu_{\bar g}}{V_{\bar g}(B_r^{\bar g}(0))}.$$
Making use of Lemma  \ref{lem: |l(x)-bar l(x)|}, one immediately concludes:
\begin{equation}\label{V_bar B minus A_delta to 0}
	\lim_{r\to\infty}\frac{V_{\bar g}(B_r^{\bar g}(0)\setminus A_{\delta, r})}{V_{\bar g}(B_r^{\bar g}(0))}=0.
\end{equation}
Next we observe that the two inclusion relations \eqref{Omega_s subset B_s} and \eqref{Adelta subet Omega 1+delta r} imply the following estimate:
\begin{align*}
	&\frac{V_{\bar g}(B_{(1+\delta)r}^g(0))}{|\mathbb{B}^n|((1+\delta)r)^n}\\	\geq &\frac{V_{\bar g}(\Omega_{(1+\delta)r})}{|\mathbb{B}^n|((1+\delta)r)^n}\\
	\geq &\frac{V_{\bar g}(A_{\delta, r})}{|\mathbb{B}^n|((1+\delta)r)^n}\\
	=&\frac{1}{(1+\delta)^n}\left(1-\frac{V_{\bar g}(B_r^{\bar g}(0)\setminus A_{\delta,r})}{V_{\bar g}(B_r^{\bar g}(0))}\right)\cdot \frac{V_{\bar g}(B_r^{\bar g}(0))}{|\mathbb{B}^n|r^n}.
\end{align*}
By taking $r$ goes to $\infty$,  making use of the limit \eqref{V_bar B minus A_delta to 0}, together with Lemmas \ref{lem: Vbarg} and \ref{lem:V_g Omega_r vs V_ bar g}, one has
$$\liminf_{r\to\infty}\frac{V_g(B_r^g(0))}{|\mathbb{B}^n|r^n}\geq (1+\delta)^{-n}(1-\alpha_0)^{n-1}.$$
Since the constant $\delta$ is arbitrary, and left hand side is independent of $\delta$,  let $\delta $ go to $0$ to obtain the inequality as we expected. Thus the proof of Theorem \ref{lem:volume ratio and Q-curvature} is complete. \qed

\section{Proof of main theorems and further discussions}\label{sec:proof of main theorem}

The main purpose of this section is to prove our Theorem \ref{thm:main theorem}. Basically our argument is based on our previous various estimates and together with several existing results, including Brendle's inequality \eqref{brendle's inequality}, Chang-Qing-Yang's identity \eqref{Finn for Q},  Lemma \ref{lem:volume ratio and Q-curvature} and Theorem \ref{thm:positive sectional for normal metric} together. Let us present it in detail as follows.

\vskip .2in

{\bf Proof of Theorem \ref{thm:main theorem}:}

\vskip .1in

Under the assumption that the top $Q$-curvature is non-negative, by our Theorem \ref{thm:positive sectional for normal metric}, we know that the sectional curvature of $g$ is non-negative. Consequently, the  Ricci curvature $g$ is also non-negative.  Thus, by Bishop-Gromov volume comparison theorem and Theorem \ref{lem:volume ratio and Q-curvature}, we obtain that 
$$\mathrm{AVR}(g)\geq (1-\alpha_0)^{n-1}.$$
For any bounded domain $\Omega$ with smooth boundary $\partial \Omega$,  making use of this estimate and  Brendle's isoperimetric inequality \eqref{brendle's inequality} (See Corollary 1.3 in \cite{Brendle}), we obtain
$$|\partial \Omega|_g^{\frac{n}{n-1}}\geq n^{\frac{n}{n-1}}|\mathbb{B}^n|^{\frac{1}{n-1}}(1-\alpha_0)|\Omega|_g.$$
From  this estimate and the definition of $I_g$ (See the definition \eqref{I_g def}), it follows that 
$$I_g\geq 1-\alpha_0.$$
On the other hand,  the generalized Chang-Qing-Yang's identity \eqref{Finn for Q} gives the upper bound
$$I_g\leq 1-\alpha_0.$$
Therefore, we conclude that 
$$I_g=1-\alpha_0.$$
This completes the proof. \qed

\vspace{2em}

Combining our Theorem \ref{lem:volume ratio and Q-curvature} with Theorem 1.2 in \cite{BK}, one has the following sharp Sobolev inequality on $(\mr^n, g)$ with $g$ as in Theorem \ref{thm:main theorem}. For $1<p<n$, we denote its $(1, p)$ Sobolev space by $W^{1, p}(M)$ with  $$\dot{W}^{1,p}(M,g):=\{U\in L^{\frac{np}{n-p}}(M,g):|\nabla_g U|\in L^p(M,g)\}.$$
\begin{corollary}
	Let $g=e^{2u}|dx|^2$ be  a smooth complete normal metric  on $M=\mr^n$. Assume the dimension of $M$  is at least two and the top $Q$-curvature  $Q_g^{(n)}$ is non-negative.   If $p\in (1,n)$, then for any $U\in \dot{W}^{1,p}(\mr^n,g)$, one has
	$$(1-\alpha_0)^{\frac{n-1}{n}}\left(\int_{\mr^n}|U|^{\frac{np}{n-p}}\ud\mu_g\right)^{\frac{n-p}{np}}\leq AT(n,p)\left(\int_{\mr^n}|\nabla_g U|^p\ud\mu_g\right)^{\frac{1}{p}}$$
	where $AT(n,p)$ is the sharp Sobolev constant for Euclidean metric.
\end{corollary}

\vspace{2em}

We now turn to the question when above inequality can be equality in Theorem \ref{lem:volume ratio and Q-curvature}.
\begin{prop}
		For a smooth complete  normal metric $g=e^{2u}|dx|^2$ on $\mr^n$, if its dimension  $n\geq 2$ and the top $Q$-curvature $Q^{(n)}_g$ is non-negative, then the volume of $g$ grows like the volume of ball in Euclidean space with precise ratio: 
		\begin{equation*}
			\lim_{r\to\infty}\frac{V_g(B_r^g(0))}{|\mathbb{B}^n|r^n}=(1-\alpha_0)^{n-1}.
		\end{equation*}
\end{prop}
\begin{proof}
Again, making use of Theorem \ref{thm:positive sectional for normal metric},  the Ricci curvature of $g$ is non-negative. Then, apply  Bishop-Gromov volume comparison theorem to have 
\begin{equation}\label{BG volume limit}
	\lim_{r\to\infty}\frac{V_g(B_r^g(0))}{|\mathbb{B}^n|r^n}=\mathrm{AVR}(g).
\end{equation}
Then our Theorem \ref{thm:main theorem} and Brendle's isoperimetric inequality \eqref{brendle's inequality} can be applied to obtain
$$(1-\alpha_0)^{n-1}\geq 	\mathrm{AVR}(g).$$
This estimate, together with Lemma \ref{lem:volume ratio and Q-curvature} and \eqref{BG volume limit}, shows that our claim holds true.  Thus we complete the proof.
\end{proof}
\begin{remark}
	In fact, such a result  can also be obtained by combining Theorem \ref{thm:positive sectional for normal metric} and Theorem 1.4 in \cite{Li conformal} with Theorem 1.2 in \cite{Ma}.
\end{remark}

The following lemma is simple yet important, as it provides another upper bound for  $I_g$.
\begin{lemma}\label{lem:I_g leq 1}
	 Let $g=e^{2u}|dx|^2$  on $\mr^n$ be a smooth conformal metric  with  dimension $n\geq 2.$ Then, there holds
	$$I_g\leq 1.$$
\end{lemma}
\begin{proof}
	Let  $B_r(0)$ be a Euclidean ball with $r > 0$. Since $u$ is smooth, it is not hard to check that 
	\[\frac{(\int_{\partial B_r(0)} e^{(n-1)u}\ud\sigma)^{\frac{n}{n-1}}}{n^{\frac{n}{n-1}}|\mathbb{B}^n|^{\frac{1}{n-1}}\int_{B_r(0)}e^{nu}\ud x} = \frac{(\fint_{\partial B_r(0)} e^{(n-1)u}\ud\sigma)^{\frac{n}{n-1}}}{\fint_{B_r(0)} e^{nu} \ud x}  \to 1, \quad \mathrm{as}\; r\to 0.\]
	Then our claim follows from the definition of  $I_g$ (See definition \eqref{I_g def}).  Thus we complete the proof.
\end{proof}

\vspace{2em}

{\bf Proof of Theorem \ref{thm:Q_g leq0}:}
From Theorem \ref{thm:positive sectional for normal metric}, we have 
$$\sec_g\leq 0.$$
If the Cartan-Hadamard conjecture holds, we can apply it to obtain
$$I_g\geq 1.$$
Combine  this with Lemma \ref{lem:I_g leq 1} to complete the proof of Theorem \ref{thm:Q_g leq0}. \qed

\vspace{3em}

As  mentioned in the introduction, Huber \cite{Huber 54} extended Fiala's sharp inequality \eqref{Fiala's inf} to obtain  the sharp  inequality \eqref{Fiala-Huber inequality}.  A natural question is if such an inequality can be generalized to higher dimensions. We believe this is likely the case and leave it as a conjecture.

\begin{conjecture}
	Let $g=e^{2u}|dx|^2$ be a smooth complete normal metric  on $\mr^n$ with finite total $Q$-curvature  and dimension $n$  at least three. Then, the isoperimetric ratio $I_g$ should be given as
	$$I_g= 1-\frac{2}{(n-1)!\,|\mathbb{S}^n|}\int_{\mathbb{R}^n} \bigl(Q^{(n)}_g\bigr)^+ \ud\mu_g.$$
\end{conjecture}
\begin{remark}
In the case where $Q^{(n)}_g$ is non-negative or non-positive, and assuming the Cartan–Hadamard conjecture, the above conjecture follows from Theorems \ref{thm:main theorem} and \ref{thm:Q_g leq0}.
\end{remark}

\end{document}